\documentclass[10pt,a4paper, leqno]{article}
\usepackage[utf8]{inputenc}
\usepackage{amsmath}
\usepackage{amsfonts}
\usepackage{amssymb}
\usepackage{comment}
\usepackage{hyperref}
\usepackage{enumitem,pstricks,xy,pst-node}
\xyoption{all}
\usepackage{amsthm}
\usepackage{booktabs} 
\usepackage[color]{changebar} 
\cbcolor{red} 
\usepackage{mathrsfs} 
\usepackage{verbatim} 
\usepackage[disable] 
{todonotes}
\usepackage[normalem]{ulem} 
\usepackage{caption} 
\usepackage{nicefrac} 
\usepackage{enumitem} 
\usepackage{bbold} 
\usepackage{tikz} 
\usepackage{accents} 

\usepackage{todonotes}
\usepackage{ytableau}

\newtheorem{theorem}{Theorem}[section]
\newtheorem{lemma}[theorem]{Lemma}
\newtheorem{proposition}[theorem]{Proposition}
\newtheorem{corollary}[theorem]{Corollary}
\newtheorem{definition}[theorem]{Definition}
\theoremstyle{definition}

\newtheorem{remark}[theorem]{Remark}
\newtheorem*{ack}{Acknowledgements}

\theoremstyle{remark}
\newtheorem{example}[theorem]{Example}

\newcommand{\PP}{\mathbb{P}}

\def\cU{{\mathcal U}}
\def\cM{{\mathcal M}}

\def\zero{\mathscr{Z}}

\def\PP{\mathbf P}

\def\CC{\mathbb{C}}

\def\ZZ{\mathbb{Z}}

\def\cV{{\mathcal V}}

\def\cO{{\mathcal O}}

\def\af1{\mathbf{aff}_1}

\def\comma{{,}}

\DeclareMathOperator{\Aut}{Aut}

\DeclareMathOperator{\Sym}{Sym}

\DeclareMathOperator{\codim}{codim}
\DeclareMathOperator{\Pic}{Pic}

\DeclareMathOperator{\HHH}{H}

\DeclareMathOperator{\Ker}{Ker}
\DeclareMathOperator{\Gr}{Gr}
\DeclareMathOperator{\IGr}{IGr}
\DeclareMathOperator{\I2Gr}{I_2Gr}
\DeclareMathOperator{\ItriGr}{I_3Gr}

\DeclareMathOperator{\GL}{GL}
\DeclareMathOperator{\Sp}{Sp}
\DeclareMathOperator{\PGL}{PGL}
\DeclareMathOperator{\SL}{SL}
\DeclareMathOperator{\diag}{diag}
\DeclareMathOperator{\pt}{pt}
\DeclareMathOperator{\modulo}{mod}

\newcommand{\ladi}{\begin{lastadd}}
\newcommand{\ladf}{\end{lastadd}}
\newcommand{\lrei}{\begin{lastrem}}
\newcommand{\lref}{\end{lastrem}}
\newenvironment{lastadd}
{\cbstart\color{red}}
{\todo{red to remove}\cbend}
\newenvironment{lastrem}
{\cbstart\color{yellow}}
{\cbend}

\author{Vladimiro Benedetti\thanks{Aix-Marseille Universit\'e, CNRS, Centrale Marseille, I2M, UMR 7373, 13453 Marseille, France.}}
\title{Bisymplectic Grassmannians of planes}
\begin{document}
\maketitle

\begin{abstract}
The bisymplectic Grassmannian $\I2Gr(k,V)$ parametrizes $k$-dimensional subspaces of a vector space $V$ which are isotropic with respect to two general skew-symmetric forms; it is a Fano variety which admits an action of a torus with a finite number of fixed points. In this work we study its equivariant cohomology when $k=2$; the central result of the paper is an equivariant Chevalley formula for the multiplication of the hyperplane class by any Schubert class. Moreover, we study in detail the case of $\I2Gr(2, \CC^6)$, which is a quasi-homogeneous variety, we analyze its deformations and we give a presentation of its cohomology.
\end{abstract}
\setcounter{tocdepth}{1}

\section{Introduction}


In complex algebraic geometry, classical Grassmannians are a special kind of homogeneous spaces for classical groups. They have been studied thoroughly for more than a century from different point of views: their geometry is governed by a rich combinatorial description, which manifests itself in many classical results about their cohomology. Moreover, the homogeneity condition has been very useful to investigate further properties of these varieties, such as their equivariant and quantum cohomology (see for instance \cite{KTpuzzles}, \cite{Buchqcohom}, \cite{Tamvakis}, \cite{Buchtam}). Among classical Grassmannians, symplectic (respectively orthogonal) ones parametrize subspaces of a given vector space which are isotropic with respect to a non-degenerate symplectic (resp. orthogonal) form. 


Even for varieties which admit an action of a \emph{sufficiently big} algebraic group, when the homogeneity hypothesis is dropped less is known: some efforts have led to the notion of GKM varieties (for the action of tori with a finite number of zero and one dimensional orbits, they are defined in \cite{GKM}) and some results have been obtained for specific examples (for instance, see \cite{Pech}, \cite{MiSh2018} and \cite{GPPS}). In this paper we present a work on a particular class of varieties, called bisymplectic Grassmannians, which are not homogeneous but admit an action of a \emph{big} torus. 

In general, one can define multisymplectic (respectively multiorthogonal) Grassmannians as the varieties parametrizing subspaces of a given vector space which are isotropic with respect to a fixed number of general symplectic (resp. orthogonal) forms. As an example, consider the unique Fano threefold of degree $22$, which is usually denoted by $\cV_{22}$, and that appears in Iskovski's classification; Mukai showed that it can be seen as a trisymplectic Grassmannian $\ItriGr(3,7)$ of $3$-dimensional subspaces of $\CC^7$.


Of course, in general, asking the isotropy condition with respect to many symplectic forms implies that the corresponding Grassmannian is no longer homogeneous. However, in the case of bisymplectic Grassmannians (two symplectic forms, denoted by $\I2Gr(k,2n)$) and of orthosymplectic Grassmannians (one symplectic and one orthogonal form), one can prove that it is still possible to define an action of a torus $T$ with a finite number of fixed points. Moreover, for extremal values of $k$, the bisymplectic Grassmannian is actually a homogeneous variety: $\I2Gr(1,2n)\cong \PP^{2n-1}$ and $\I2Gr(n,2n)\cong (\PP^1)^n$ (for the second isomorphism, which is a priori quite surprising, see \cite{Kuz}). Therefore, even though $\I2Gr(k,2n)$ is not homogeneous when $k\neq 1, n$ (consequence of the fact that it has non-trivial deformations, see Theorem \ref{thmsmalldefbisym}), one may still expect it to behave quite similarly to homogeneous spaces. 


However, this non-homogeneity implies that some difficulties appear when trying to study the $T$-equivariant cohomology of $\I2Gr(k,2n)$. In this paper we show how to determine the equivariant cohomology of bisymplectic Grassmannians of planes, i.e. for $\I2Gr(2,2n)$ (when $k=2$). This variety has a simple geometric construction: it can be seen as the intersection of two hypersurfaces in $\Gr(2,2n)$. Even so, the determination of its equivariant cohomology is an interesting problem for different reasons: on one hand, as already remarked, we can apply some equivariant tools in a rather simple non-homogeneous situation; on the other hand, we believe that the proofs of the results we state here for $\I2Gr(2,2n)$ can be adequately generalized in the case of bisymplectic Grassmannians $\I2Gr(k,2n)$ with $k\neq 2$. We intend to analyse this more general situation in the future.

The main results we obtain for $\I2Gr(2,2n)$ concern its equivariant cohomology. Firstly, we show that the classes of an additive basis of the cohomology can be uniquely determined by a finite number of relations coming from $T$-equivariant curves (Theorem \ref{thmuniShckeqtwo}); these classes correspond to the \emph{Schubert} subvarieties that appear in the Bialynicki-Birula decomposition. Then, we find an equivariant Chevalley formula for the multiplication of any class with the hyperplane class (Theorem \ref{thmChevform}), from which one can recover the corresponding formula for the classical cohomology. As a result, one can compute the classes of Schubert varieties inductively (Corollary \ref{thm1coho22n}).


As an application, we give an explicit presentation of the cohomology of $\I2Gr(2,6)$. This bisymplectic Grassmannian is particularly interesting because it is quasi-homogeneous: it admits an action of $\SL(2)^3$ with a dense affine orbit. Moreover, even though it has no smooth deformations, we are able to describe all its singular flat deformations (Proposition \ref{HilbI2Gr26}).


The structure of the paper is as follows. In the first part we recall general results about bisymplectic Grassmmannians, whose detailed proofs can be found in \cite{PhDBen}. We also recall some facts about symplectic Grassmannian, as they are useful to understand our situation better. In the central part of the paper we deal with bisymplectic Grassmannians of planes $\I2Gr(2,2n)$; after recalling some basic properties of the equivariant cohomology, we prove the two main results of the paper: the unicity for equivariant Schubert classes in Theorem \ref{thmuniShckeqtwo} and the equivariant Chevalley formula in Theorem \ref{thmChevform}. Finally, we study in detail the quasi-homogeneous variety $\I2Gr(2,6)$, we determine its orbit structure, its flat deformations and we give a presentation of its classical cohomology.


\begin{ack}
This work has been carried out in the framework of the Labex Archim\`ede (ANR-11-LABX-0033) and 
of the A*MIDEX project (ANR-11-IDEX-0001-02), funded by the ``Investissements d'Avenir" 
French Government programme managed by the French National Research Agency.\\
I would like to thank my advisor Laurent Manivel for presenting the problem to me, and for guiding the study with numerous and insightful questions.
\end{ack}

\section{Bisymplectic Grassmannians}

In this section we recall some basic definitions and facts about bisymplectic Grassmannians. The content of what follows can be found in \cite[Chapter 4]{PhDBen}, therefore we will omit the proofs. Introducing the notations for general bisymplectic Grassmannians is useful for two reasons. On one hand, it constitutes the natural framework in which to study the Grassmannians of planes, which can be seen as a particular example. On the other hand, it allows to compare what can be done in our particular example with the general situation; indeed, we believe that the ideas developed in this paper can be used fruitfully to obtain analogous results for general bisymplectic Grassmannians, which we intend to do in the next future.

Let us consider the Grassmannian $\Gr(k,2n)$ of $k$-dimensional subspaces inside a vector space of dimension $2n$. From now on, if not otherwise stated, we will assume that $2\leq k\leq n$. By fixing a skew-symmetric form $\omega$ over $\CC^{2n}$, one can consider the subvariety $\IGr(k,2n)$ inside $\Gr(k,2n)$ of isotropic subspaces with respect to $\omega$. If $\omega$ is non-degenerate, $\IGr(k,2n)$ is smooth, and it is a rational homogeneous variety for the natural action of $\Sp(2n)\subset \GL(2n)$. Denoting by $\cU$ the tautological bundle over the Grassmannian, the variety $\IGr(k,2n)$ can be seen as the zero locus of a general section of $\wedge^2 \cU^*$ over $\Gr(k,2n)$; indeed notice that, by the Borel-Weil Theorem, $\HHH^0(\Gr(k,2n),\wedge^2 \cU^*)\cong \wedge^2 (\CC^{2n})^*$. We will refer to $\IGr(k,2n)$ as the isotropic (or symplectic) Grassmannian.

Let us now fix two skew-symmetric forms $\omega_1$, $\omega_2$ over $\CC^{2n}$. 

\begin{definition}
The bisymplectic Grassmannian is the subvariety $\I2Gr(k,2n)$ inside $\Gr(k,2n)$ of subspaces isotropic with respect to $\omega_1$ and $\omega_2$. Equivalently, the points in $\I2Gr(k,2n)$ are isotropic with respect to the pencil $\langle \omega_1,\omega_2\rangle$. 
\end{definition}

\begin{remark}
As we will see later, there is not only one isomorphism class of bisymplectic Grassmannians. Indeed, the definition depends on the choice of a pencil $\langle \omega_1,\omega_2\rangle$. However, we will still refer to \emph{the} bisymplectic Grassmannian in the following.
\end{remark}

Of course, $\I2Gr(k,2n)\subset \IGr(k,2n)_i$, where $\IGr(k,2n)_i$ is the symplectic Grassmannian with respect to $\omega_i$, $i=1,2$. The fact that $\I2Gr(k,2n)$ is not empty is a consequence of the fact that $\I2Gr(n,2n)\neq \emptyset$ (see Example \ref{exkeqn}). Moreover, $\I2Gr(k,2n)$ can be seen as the zero locus of a section of $(\wedge^2 \cU^*)^{\oplus 2}$ over $\Gr(k,2n)$; by Bertini's theorem, if this section is general, the bisymplectic Grassmannian is smooth. In this case, its dimension is $2k(n-k)+k$ and, by the adjunction formula, its canonical bundle is 
\[
K_{\I2Gr(k,2n)}=\cO(-2n+2k-2);
\] 
therefore, $\I2Gr(k,2n)$ is a Fano variety. In the next sections, we will study under which conditions the bisymplectic Grassmannians are smooth (i.e. for what kind of pencils). Before doing so, let us deal with the case $k=n$.

\begin{example}[$k=n$] 
\label{exkeqn}
In \cite{Kuz}, Kuznetsov proves that the variety $\I2Gr(n,2n)$ is smooth exactly when the pencil $\langle \omega_1,\omega_2\rangle$ intersects the Pfaffian divisor $D\subset \PP(\wedge^2 (\CC^{2n})^*)$ (of degree $n$) in $n$ distinct points; in this case, the two forms are simultaneously block diagonalizable (with blocks of size $2\times 2$), and there exists an isomorphism 
\begin{equation}
\label{isomKuzbisymn}
\I2Gr(n,2n)\cong (\PP^1)^n.
\end{equation}
Therefore, the automorphism group of $\I2Gr(n,2n)$ is $(\PGL(2))^n \times \mathfrak{S}_n$ (where $\mathfrak{S}_n$ is the group of permutations of $n$ elements). Surprisingly enough, from the isomorphism one realises that $\I2Gr(n,2n)$ has no small deformations.
\end{example}

\begin{example}[$k=2$]
The bisymplectic Grassmannian of planes is $\I2Gr(2,2n)$. It is just the intersection of two hyperplane sections in $\Gr(2,2n)$. However, we will see how considering it as a particular case of $\I2Gr(k,2n)$ is an effective point of view.
\end{example}

From now on, the zero locus of a section $s$ will be denoted by $\zero(s)$. Moreover, let us denote by $V=\CC^{2n}$.

\subsection{Small deformations}
\label{secsmooth}

As a consequence of Example \ref{exkeqn}, one may wonder whether all bisymplectic Grassmannians admit no small deformations, and what is their automorphism group. We address these questions in the following. In order to do so, we need to state a result on the normal form of a pencil of skew-symmetric forms which defines a smooth bisymplectic Grassmannian. Let $D\subset \PP(\wedge^2 V^*)$ be the Pfaffian divisor of degree $n$.

\begin{proposition}[\cite{PhDBen}]
\label{propsmooth}
Let $\Omega=\langle \omega_1,\omega_2\rangle \subset \PP(\wedge^2 (V^*))$ be a pencil of skew-symmetric forms such that $\zero(\Omega)\subset \Gr(k,V)$ has the expected dimension. If $\zero(\Omega)$ is smooth then $\Omega\cap D={p_1,\dots,p_n}$, where the $p_i$'s are $n$ distinct points such that:
\begin{enumerate}
\item $\dim(\Ker(p_i))=2$ for $1\leq i\leq n$;
\item $V=\Ker(p_1)\oplus\dots \oplus \Ker(p_n)$.
\end{enumerate}
\end{proposition}

The proof of this result follows the same line of ideas as the analogous one used in \cite{Kuz}. From now on, if not otherwise stated, we will assume that the bisymplectic Grassmannians are smooth. We will denote by $K_i=\Ker(p_i)$.

\begin{remark}
\label{remdiagforms}
The proof of the previous proposition (see \cite{PhDBen}) actually shows that if $\zero(\Omega)$ is smooth then all the forms in $\Omega$ are simultaneously block diagonalizable. Moreover, as any non-degenerate form is conjugate to the standard one, one can suppose that $\Omega$ is generated by $\omega_1$ and $\omega_2$ with:
\[
\omega_1=\sum_{i=1}^{n} x_i\wedge x_{-i},
\]
\[
\omega_2=\sum_{i=1}^{n} \lambda_i x_i\wedge x_{-i},
\]
where $\langle x_i,x_{-i}\rangle = (K_i)^*$ for $1\leq i\leq n$, and the $\lambda_i$'s are all distinct.
\end{remark}

The following result answers the two questions asked at the beginning of this section:

\begin{theorem}[\cite{PhDBen}]
\label{thmsmalldefbisym}
The following isomorphisms hold:
\[
\HHH^0(X,T_X)\cong \mathfrak{sl}(2)^n,
\]
\[
\HHH^1(X,T_X)\cong \CC^{n-3}.
\]
\end{theorem}

\begin{remark}
The fact that $\HHH^0(X,T_X)\cong \mathfrak{sl}(2)^n$ should not be surprising; indeed, by Proposition \ref{propsmooth} we know that the forms in $\Omega$ can be simultaneously block diagonalized, the blocks being the $2$-dimensional subspaces $K_i$. A consequence of this is the fact that, for $1\leq i\leq n$, the group $\PGL(K_i)\subset PGL(\PP(\wedge^2 V^*))$ fixes the pencil $\Omega$. Therefore, it is contained in the automorphism group of $\zero(\Omega)=X$. The fact that these are the only automorphisms of $X$ modulo a finite group is a consequence of the previous theorem. To state it more intrinsically, we can write:
\[
T_{\Aut(X)}\cong \HHH^0(X,T_X)\cong \mathfrak{sl}(K_1)\oplus\dots \oplus \mathfrak{sl}(K_n).
\]
Moreover, this observation implies that a $n$-dimensional torus acts on $X$, which we will use later on.
\end{remark}

\begin{remark}
\label{remdefquasihomex}
When $n=3$ (and $k=2$), the variety $X$ has no small deformations. Moreover, by Proposition \ref{propsmooth}, if $\zero(\Omega)$ is smooth, $\Omega$ intersects the Pfaffian divisor $D$ in three points $p_1,p_2,p_3$; by changing coordinates if necessary, we can suppose that $p_1=[x_1\wedge x_{-1}+x_2\wedge x_{-2}]$ and $p_2=[x_2\wedge x_{-2}+x_3\wedge x_{-3}]$, where $(x_{\pm 1},x_{\pm 2},x_{\pm 3})$ is a basis of $V^*$. Thus, there is only one smooth isomorphism class of $\I2Gr(2,6)$. In Section \ref{i2gr26} we will study more in detail this variety, its flat deformations and its (equivariant) cohomology.
\end{remark}



What the theorem tells us is that the moduli stack $\cM_{bisym(k,n)}$ of bisymplectic Grassmannians should have dimension $n-3$. This is the same as the dimension of the moduli stack $\cM_{n}$ of $n$ points inside $\PP^1$. It is straightforward to see that there is a dominant rational morphism
\[
\pi: \Gr(2, \wedge^2 V^*)/PGL(V) \dashrightarrow \cM_{bisym(k,n)},
\]
where $\Gr(2, \wedge^2 V^*)/PGL(V)$ is the GIT quotient. This quotient has dimension $n-3$, i.e. it is not of the expected dimension. Indeed, on the open subset inside $\Gr(2, \wedge^2 V^*)$ of diagonalizable pencils, each point $\Omega$ is fixed by a copy of $\SL(2)^n\subset PGL(V)$. Moreover, one can prove that there is a birational morphism
\[
\Gr(2, \wedge^2 V^*)/PGL(V) \dashleftarrow\dashrightarrow \cM_{n}.
\]



However, in order to have a birational model of $\cM_{bisym(k,n)}$, one should understand the degree of $\pi$; as it will not be needed in the following, we leave this question open for further work. 

\subsection{The torus action on $\I2Gr(k,V)$}
\label{secBBdecomposition}

The variety $\I2Gr(k,V)$ admits an action of a torus with a finite number of fixed points. We summarize here the first consequences of the existence of this action. This will allow us to introduce some useful notation.

Let $\I2Gr(k,V)$ be defined by the forms $\omega_1$ and $\omega_2$ described in Remark \ref{remdiagforms}, and let $\IGr(k,V)$ be the symplectic Grassmannian defined by $\omega_1$ which contains $\I2Gr(k,V)$. Moreover let $T\cong (\CC^*)^n$ be the maximal torus inside $\Sp(V)$ which is contained inside $\SL(2)^n\subset \Aut(\I2Gr(k,V))$. For simplicity, we will assume from now on that $T$ is the diagonal torus $\diag(t_n,\dots,t_1,t_1^{-1},\dots,t_n^{-1})\subset \Sp(V)$. It acts on $\IGr(k,V)$ with a finite number of fixed points, and as a consequence the induced action on $\I2Gr(k,V)$ has a finite fixed locus as well. The surprising fact is that the two fixed loci are the same:

\begin{proposition}[\cite{PhDBen}]
\label{propfinpoints}
There are $2^k {n\choose k}$ fixed points for the action of $T$ on $\IGr(k,V)$ and on $\I2Gr(k,V)$. They are parametrized by subsets $I\subset \{\pm 1,\dots , \pm n\}$ such that $I\cap (-I)=\emptyset$.
\end{proposition}

If $V=\langle v_n\dots,v_1,v_{-1},\dots,v_{-n} \rangle$, with $K_i=\langle v_i,v_{-i}\rangle$, then the fixed point corresponding to a subset $I=(i_1,\dots ,i_k)$ is given by $p_I=[v_I]=[v_{i_1}\wedge \dots \wedge v_{i_k}]$.

\begin{definition}
We will say that a subset $I\subset \{\pm1,\dots,\pm n\}$ is admissible if $I\cap (-I)=\emptyset$.
\end{definition}

Therefore, by the Bialynicki-Birula decomposition (see \cite{BBdec}), by fixing a general one dimensional torus $\tau\subset T$, we can associate to each fixed point $p_I$, where $I\subset \{\pm 1,\dots , \pm n\}$ is admissible, a \emph{Schubert} variety $\sigma_I\subset \I2Gr(k,V)$, which is the closure of a \emph{Schubert} cell isomorphic to an affine space (the terminology is borrowed from the homogeneous situation). The Schubert cell is defined as the set of points which accumulate towards $p_I$ under the action of $\tau$. The condition that $\tau$ needs to satisfy in order to give the decomposition is that it acts with a finite number of fixed points. For instance, let 
\begin{equation}
\label{eqdeftau}
\tau=\diag(t^n,\dots,t,t^{-1},\dots,t^{-n})\subset T.
\end{equation}

\begin{lemma}
\label{lemfixtau}
The one dimensional torus $\tau$ acts with a finite number of fixed points over $\IGr(k,V)$ and over $\I2Gr(k,V)$.
\end{lemma}

\begin{remark}
\label{remorbitBSchubert}
The symplectic Grassmannian $\IGr(k,V)$ is a homogeneous variety under the action of $\Sp(V)$, and as such it has a natural Bruhat decomposition in orbits under the action of a Borel subgroup of $\Sp(V)$. It turns out that the Bruhat decomposition and the Bialynicki-Birula one are the same (see \cite{BB2dechom}[Book $II$, example $4.2$]). We will denote by $\sigma'_I$ the Schubert varieties of $\IGr(k,V)$, and by $B$ a Borel subgroup of $\Sp(V)$.

The identification of the two decompositions implies that if a fixed point $p_J$ belongs to a Schubert variety $\sigma_I'$, then actually $\sigma_J'=\overline{B.p_J}\subset\overline{B.p_I}= \sigma_I'$. This fact is crucial when trying to compute the equivariant cohomology of $\IGr(k,V)$, as we will see. However, this property will not hold in the bisymplectic case, and it is one of the main reasons why computing the equivariant cohomology for $\I2Gr(k,V)$ becomes more difficult. 
\end{remark}

As $p_I$ is fixed, the torus $T$ acts on the vector space $T_I:=T_{\I2Gr(k,V),p_I}$. Let $\epsilon_i\in \Xi (T)$ be the character of $T$ given by $\diag(t_n,\dots,t_1,t_1^{-1},\dots,t_n^{-1})\mapsto t_i$. If $i<0$, we denote by $\epsilon_i$ the character $-\epsilon_{-i}$.

\begin{lemma}
\label{lemweightsbisym}
The weights of the action of $T$ on $T_I$ are
\[
-2\epsilon_i \mbox{ for } i\in I \mbox{ and }
\]
\[
\epsilon_i -\epsilon_j \mbox{ for } i\notin I\cup (-I),\,\,\, j\in I.
\]
\end{lemma}

The weights of the action of $\tau$ are easily deduced from Lemma \ref{lemweightsbisym}; indeed, under the identification $\Xi(\tau)\cong \ZZ$, it is sufficient to notice that $\epsilon_i \mapsto i$ under the morphism $j^*:\Xi(T)\to \Xi(\tau)$ induced by the natural inclusion $j:\tau \to T$. The Schubert variety $\sigma_I$ is smooth at $p_I$ and the tangent space $T_{\sigma_I,p_I}$ is the $\tau$-invariant subspace of $T_I$ whose weights with respect to $\tau$ are negative.  

\begin{definition}
From now on, we will say that $\xi\in \Xi(T)$ is $\tau$-positive (and we will denote it by $\xi>0$) if $j^*(\xi)>0$, and $\tau$-negative if $j^*(\xi)<0$.
\end{definition}

Therefore, given a certain subset $I$, it is possible to compute the codimension of $\sigma_I$ as:
\[
\codim(\sigma_I)=\#\{(i,j)\mbox{ s.t. }i\notin I\cup (-I)\mbox{ , }j\in I, \mbox{ and }j>i\}+\#\{j\in I\mbox{ s.t. }j< 0\}.
\]

The decomposition of $\I2Gr(k,V)$ in Schubert cells isomorphic to an affine space implies that the cohomology of $\I2Gr(k,V)$ as a $\ZZ$-module is generated by the classes $\overline{\sigma_I}$ for $I$ admissible. The odd Betti numbers are therefore all equal to zero. Let $\{b^i_{k,n}\}_i$ be the even Betti numbers of $\I2Gr(k,2n)$ (where $i$ is the codimension). We will denote by $S_{k,n}$ the sequence of integers:
\[
S_{k,n}=(b^0_{k,n},\dots,b^{\dim(\I2Gr(k,2n))}_{k,n},0,\dots,0,\dots).
\]
Of course, the decomposition of $\I2Gr(k,2n)$ in Schubert cells whose closures are the $\sigma_I$'s implies that $b^i_{k,n}$ is equal to the number of subsets $I$ such that $\codim(\sigma_I)=i$. We will denote by $[h]$ the shift on the right by $h$. For instance, $S_{k,n}[1]=(0,b^0_{k,n},\dots,b^{\dim(\I2Gr(k,2n))}_{k,n},0,\dots,0,\dots)$.

\begin{theorem}[\cite{PhDBen}]
The following recursive formula holds for the Betti numbers of $\I2Gr(k,2(n+1))$:
\begin{equation}
\label{eqbettinumbers}
S_{k,n+1}=S_{k,n}[k]+S_{k-1,n}+S_{k-1,n}[1+2(n+1-k)].
\end{equation}
\end{theorem}

\begin{example}
We give here a list of examples of Betti numbers of bisymplectic Grassmannians for small $k,n$:
\begin{itemize}
\item[] $S_{2,3}=(1,1,2,4,2,1,1,0,\dots)$;
\item[] $S_{3,4}=(1,1,2,6,6,6,6,2,1,1,0,\dots)$;
\item[] $S_{2,4}=(1,1,2,2,3,6,3,2,2,1,1,0,\dots)$.
\end{itemize}
\end{example}

\begin{remark}
The case of the Grassmannian of planes is particularly easy because $\I2Gr(2,2n)$ is a codimension $2$ complete intersection inside $\Gr(2,2n)$; all its Betti numbers except the middle term can be derived from those of $\Gr(2,2n)$ (or of $\IGr(2,2n)$) by applying Lefschetz hyperplane theorem. Moreover, as $\chi(\IGr(2,2n))=\chi(\I2Gr(2,2n))$ (because the number of fixed points is the same for the two varieties), the middle term is the sum of the two middle Betti numbers of $\IGr(2,2n)$.
\end{remark}

\begin{remark}
By using the recursive formula, it is possible to prove that for any $1<k<n$ we have: $b^0_{k,n}=1$, $b^1_{k,n}=1$, $b^2_{k,n}=2$. In particular,
\[
\Pic(\I2Gr(k,V))\cong \ZZ.
\]
Furthermore, let $H=(n,n-1,\dots,n-k+2,n-k)$ be the subset corresponding to the codimension $1$ Schubert variety (both for $\IGr(k,V)$ and for $\I2Gr(k,V)$). Then $\sigma_H$ is a hyperplane section of $\cO(1)$ inside $\I2Gr(k,V)$, and it is a line. Indeed, it is the restriction to $\I2Gr(k,V)$ of the hyperplane section $\sigma_H'\subset\IGr(k,V)$.
\end{remark}

\subsubsection{$T$-equivariant curves}

In order to compute the $T$-equivariant cohomology of $\I2Gr(k,V)$, one needs to understand which are the $T$-invariant curves, and what are the inclusions of fixed points in Schubert varieties $p_J\in \sigma_I$. Recall that $T$-invariant curves are rational curves whose intersection with the fixed locus has cardinality $2$; these two fixed points will be denoted by $p_0$ and $p_\infty$.

\begin{lemma}
\label{lemfinTcurvessym}
There is only a finite number of $T$-invariant curves inside $\IGr(k,V)$. They are of two types:
\begin{enumerate}
\item[type $\alpha$:] curves with $p_0=p_I$ and $p_\infty=p_J$, where $\#(I\cap J)=k-1$;
\item[type $\beta$:] curves with $p_0=p_I$ and $p_\infty=p_J$, where $\#(I\cap J)=k-2$, $I- J=\{a_1,a_2\}$, $J-I=\{-a_2,-a_1\}$.
\end{enumerate}
Among these, the $T$-invariants curves which are also contained inside $\I2Gr(k,V)$ are those of type $\alpha$.
\end{lemma}

\begin{remark}
From the proof of the previous lemma in \cite{PhDBen} it is straightforward to see that the curves of type $\alpha$ are lines inside $\PP(\wedge^k V^*)$, while the curves of type $\beta$ are conics inside a $\PP^2\subset \PP(\wedge^k V^*)$.
\end{remark}

\begin{definition}
Let $I=\{a_k\geq \dots \geq a_1\}$ and $J=\{b_k\geq \dots \geq b_1\}$. If $a_i\geq b_i$ for $1\leq i\leq k$, then we will say that $I$ is greater or equal than $J$, and we will denote this by $I\geq J$.
\end{definition}

We will say that $C=C_1 \dots C_m$ is a chain of $T$-equivariant curves from $p_I$ to $p_J$ if $C_i(\infty)=C_{i+1}(0)$ for any $1\leq i\leq m$, and $C_1(0)=p_I$, $C_m(\infty)=p_J$.

\begin{lemma}[\cite{PhDBen}]
\label{leminclpoints}
For two admissible subsets $I$ and $J$, the fact that $I\geq J$ is equivalent to $p_J\in \sigma'_I\subset \IGr(k,V)$ and to the fact that there is a chain of $T$-invariant curves inside $\IGr(k,V)$ from $p_I$ to $p_J$.
\end{lemma}

\section{Equivariant cohomology of bisymplectic Grassmannians of planes}

\label{seceqcohomsymeasy}

In this section we study the $T$-equivariant cohomology of $\I2Gr(2,V)$. We begin by recalling some basic facts about equivariant cohomology. A reference for this subject is \cite{Brioneqintro}; the general results we will cite can be found in \cite{GKM} or \cite{Brioneqtorus}. Then we will analyze the case of the symplectic Grassmannian, in order to compare it with the behaviour of the bisymplectic one. The main result of this section will be a Chevalley formula for Schubert classes in $\I2Gr(2,V)$, which a priori determines inductively all the equivariant classes $\overline{\sigma_I}$ for $I$ admissible.

Let $X$ be a smooth variety on which a torus $T\cong (\CC^*)^n$ acts with finitely many fixed points $X^T=\{p_1,\dots,p_r\}$. Denote by $\Xi(T)\cong \ZZ^n$ the character group of $T$. Moreover, let $\tau\in T$ be a general $1$-dimensional torus such that its fixed locus is equal to $X^T$; then the Bialynicki-Birula decomposition for $\tau$ provides varieties $\sigma_{p_i}$ for all $1\leq i\leq r$ which are a basis for the ordinary cohomology $\HHH^*(X,\ZZ)$. 

The equivariant cohomology ring $\HHH^*_T(X)$ is an algebra over the polynomial ring $\HHH^*_T(\pt)\cong \CC[\Xi(T)]=\Sym((\Xi(T))\otimes_\ZZ \CC)$ via the push-forward map of the natural inclusion of a point $\pt$ inside $X$. An additive basis for this algebra is given by the (equivariant) classes $\overline{\sigma_{p_i}}$ for $1\leq i\leq r$.

Denote by $\HHH^*(X):=\HHH^*(X,\CC)$. The pullback map $i^*:\HHH^*_T(X)\to \HHH^*_T(X^T)$ of the natural inclusion $i:X^T\to X$ is injective; therefore 
\[
\HHH^*_T(X)= \Xi(T)\otimes_{\ZZ} \HHH^*(X)\cong \Xi(T)\otimes_{\ZZ} \bigoplus_{p_i}\CC \overline{\sigma_{p_i}}
\]
can be seen as a subring of
\[
\HHH^*_T(X^T)\cong \Xi(T)\otimes_{\ZZ} \HHH^*(X^T)\cong \Xi(T)\otimes_{\ZZ} \bigoplus_{p_i}\CC p_i\cong \CC[\Xi(T)]^{\oplus r}.
\]
Via this inclusion, we will denote by $f_{\sigma_i}\in \CC[\Xi(T)]^{\oplus r}$ the pullback of the class $\overline{\sigma_{p_i}}\in \HHH^*_T(X)$, and by $f_{\sigma_i}(p_j)=(i\circ i_j)^*\overline{\sigma_{i}}$, where $i_j:p_j\to X^T$ is the natural inclusion. Clearly, if $\epsilon_1,\dots,\epsilon_n$ is a $\ZZ$-basis of $\Xi(T)$, then $f_{\sigma_i}(p_j)\in \HHH^*_T(p_j)$ is a polynomial in $\epsilon_1,\dots,\epsilon_n$. Therefore, in order to understand the equivariant cohomology of $X$, we need to find the polynomials $f_{\sigma_i}(p_j)$. The following results hold:

\begin{theorem}
\label{thmlocgen}
The polynomials $f_{\sigma_i}(p_j)$ satisfy the following properties:
\begin{enumerate}
\item $f_{\sigma_i}(p_j)$ is a homogeneous polynomial of degree $\codim(\sigma_i)$;
\item $f_{\sigma_i}(p_j)=0$ if $p_j\notin \sigma_i$;
\item $f_{\sigma_i}(p_j)$ is the product of the $T$-characters of the normal bundle $N_{\sigma_i /X,p_j}$ whenever $\sigma_i$ is smooth at $p_j$;
\item If there exists a $T$-equivariant curve between $p_j$ and $p_k$ whose character is $\chi$, then $\chi$ divides $f_{\sigma_i}(p_j)-f_{\sigma_i}(p_k)$ for $1\leq i\leq r$.
\end{enumerate}
\end{theorem}

\begin{theorem}
\label{thmfinitecurvesgen}
If there is only a finite number of $T$-invariant curves inside $X$, then the equivariant cohomology $\HHH^*_T(X)$ is the subalgebra of $\CC[\Xi(T)]^{\oplus r}$ consisting of elements $f=(f_1,\dots,f_r)$ satisfying the last condition in Theorem \ref{thmlocgen}, i.e.:
\begin{equation}
\label{reldeteqcohom}
\begin{array}{c}\mbox{if there exists a $T$-equivariant curve between $p_j$ and $p_k$}\\
\mbox{whose character is $\chi$, then $\chi$ divides $f_j-f_k$.} 
\end{array}
\end{equation}
\end{theorem}

Moreover, from the equivariant cohomology, it is possible to recover the ordinary cohomology $\HHH^*(X)$:

\begin{theorem}
\label{thmeqclasscohomgen}
The classical cohomology $\HHH^*(X)$ can be recovered from the equivariant cohomology $\HHH^*_T(X)$ as
\[
\HHH^*(X)\cong \HHH^*_T(X)/(\epsilon_1,\cdots,\epsilon_n).
\]
\end{theorem}

Therefore, the finiteness of the number of $T$-invariant curves inside $\I2Gr(k,V)$ (Lemma \ref{lemfinTcurvessym}) and Theorem \ref{thmfinitecurvesgen} give:

\begin{theorem}
\label{thm1symgrass}
The relations in \eqref{reldeteqcohom} are enough to determine the equivariant cohomology of $\I2Gr(k,V)$.
\end{theorem}

One should be careful: being able to determine the equivariant cohomology of $\I2Gr(k,V)$ does not imply that we are able to identify the equivariant classes $\overline{\sigma_{I}}$ in general. 

\subsection{Warm-up: the homogeneous case}

In contrast to the bisymplectic case, in the homogeneous case the following proposition ensures that we can identify the equivariant classes $f_{\sigma'_I}$:

\begin{proposition}[\cite{PhDBen}]
\label{prophomvareqcohom}
Let $X=G/P$ be a homogeneous rational variety under the action of a simple group $G$. Then the maximal torus $T$ inside a Borel subgroup $B\subset G$ acts with a finite number of fixed points on $X$. Moreover, if there is only a finite number of $T$ equivariant curves, then the equivariant classes of Schubert varieties inside $\HHH_T^*(X)$ are determined by the relations $1,2,3$ in Theorem \ref{thmlocgen}.
\end{proposition}

The crucial property in order to prove this proposition is the one underlined by Remark \ref{remorbitBSchubert}. Later on we will prove the analogous result for $\I2Gr(2,V)$ by adapting the proof of the previous proposition.

We recall in the following that an equivariant Chevalley formula is known for $\IGr(k,V)$. This formula permits to compute inductively the polynomials $f_{\sigma'_I}(p_J)$. The inductive method proceeds as follows. Let us fix a Schubert variety $\sigma'_I$. Then 
\[
\mbox{if }p_J\notin \sigma'_I, \mbox{ then }f_{\sigma'_I}(p_J)=0.
\]
Moreover, $f_{\sigma'_I}(p_I)$ is just the product of the (positive) $\tau$-weights of $T_I$ (because $\sigma'_I$ is smooth at $p_I$). Notice that these two assertions are general, and will hold for the bisymplectic Grassmannian as well. 

Finally, the polynomial $f_{\sigma'_I}(\cdot)(f_{\sigma'_H}(\cdot)-f_{\sigma'_H}(p_H))$ has support over the points $p_J\in \sigma'_I$, $J\neq I$ (recall that $\sigma'_H$ is the codimension $1$ Schubert variety). By applying Lemma \ref{leminclpoints}, we obtain:
\begin{equation}
\label{eqindsym}
f_{\sigma'_I}(\cdot)(f_{\sigma'_H}(\cdot)-f_{\sigma'_H}(p_I))=\sum_{J\in I_{-1}}a_{I,J} f_{\sigma'_J}(\cdot),
\end{equation}
where $I_{-1}=\{J\mbox{ s.t. }I\geq J,\mbox{ and }\codim(\sigma'_J)=\codim(\sigma'_I)+1\}$ (notice that this is an equivariant Chevalley formula). The condition on the codimension is a consequence of the fact that $\deg(f_{\sigma'_J})=\codim(\sigma'_J)$. The coefficient $a_{I,J}$ turns out to be equal to $1$ if there is a $\alpha$-curve between $p_I$ and $p_J$, and it is equal to $2$ if there is a $\beta$-curve between $p_I$ and $p_J$. Knowing the coefficients $a_{I,J}$, one can determine inductively $f_{\sigma'_I}$ from the $f_{\sigma'_J}$'s in Equation \eqref{eqindsym}.

\subsection{Schubert classes are determined}

In this section we prove that the equivariant Schubert classes for $\I2Gr(2,V)$ are completely determined by the relations $1,2,3$ in Theorem \ref{thmlocgen}, i.e. the analogous of Proposition \ref{prophomvareqcohom}. In order to do so, we will need to understand (some) inclusions of fixed points inside Schubert varieties. In the end we will prove an equivariant Lefschetz hyperplane theorem, relating the equivariant cohomology of the symplectic Grassmannian with that of the bisymplectic one. 

\begin{remark}
From now on we will denote by $f_I(J)=f_{\sigma_I}(p_J)$. 
\end{remark}

The problem of determining the inclusions of fixed points in the case of the bisymplectic Grassmannian is more difficult to deal with. In order to understand this problem, notice that $\I2Gr(k,V)\subset \IGr(k,V)$ implies that if $p_J\in\sigma_I\subset \sigma_I'$, then $I\geq J$. Moreover $\geq$ is a partial order relation on the admissible subsets of $\{\pm 1,\dots,\pm n\}$ of cardinality $k$. We define the relation $\geq_\in$ on admissible subsets: $I\geq_\in J$ if and only if there exist admissible subsets $J=J_1,J_2,\dots,J_u=I$ such that $p_{J_i}\in \sigma_{J_{i+1}}$ for $i=1,\dots,u-1$. This relation is by construction reflexive and transitive. Moreover, it is skew-symmetric because if $I\neq J$, $I\geq_\in J$ and $J\geq_\in I$, then $J\geq I\geq J$ and $I\neq J$, which is a contradiction by the definition of $\geq$. As a result, $\geq_\in$ is a partial order relation on admissible subsets of $\{\pm 1,\dots,\pm n\}$, and as a consequence we get the following:

\begin{lemma}
There exist polynomials $a_{I,J}\in \CC[\epsilon_1,\dots,\epsilon_n]$ of degree $\codim(\sigma_J)-\codim(\sigma_I)-1$ such that
\begin{equation}
\label{eqindmethbisym}
f_I(\cdot)(f_H(\cdot)-f_H(I))=\sum_{J\in I_{\geq_\in-1}}a_{I,J} f_J(\cdot),
\end{equation}
where $I_{\geq_\in-1}=\{J\mbox{ s.t. }I\geq_\in J,\mbox{ and }\codim(\sigma_J)\leq\codim(\sigma_I)+1\}$. 
\end{lemma}

In the next section, we will use this lemma to obtain an equivariant Chevalley formula for the multiplication of Schubert varieties with $f_H$. Now just notice that in general we are looking for coefficients $a_{I,J}$'s which are not constants, but actual polynomials; determining even one of them may need the use of a lot of relations. This problem comes from the fact that, as we are in the non-homogeneous case, the fact that $p_J\in \sigma_I$ does not necessarily imply that $\sigma_J\subset \sigma_I$, or, more concretely, that $\codim(\sigma_J)>\codim(\sigma_I)$. Hence we get that the coefficients $a_{I,J}$ may very well not be constant. However, for the Grassmannians of planes, this problem can be controlled:

\begin{lemma}
\label{lemmacodimnullneg}
Suppose that $p_J\in \sigma_I$ and $\codim(\sigma_J)\leq \codim(\sigma_I)$. Then $\codim(\sigma_J)= \codim(\sigma_I)=2n-3$, $I=(i,-i+1)$ and $J=(i-1,-i)$ or $J=(i,-i-1)$.
\end{lemma}

\begin{proof}
We will prove the lemma by comparison with the symplectic Grassmannian. The weights of the action of $T$ on $T_{\IGr(2,V),p_I}$ are
\[
-2\epsilon_i \mbox{ for } i\in I \mbox{ , }
\epsilon_i -\epsilon_j \mbox{ for } i\notin I\cup (-I),\,\,\, j\in I \mbox{ and }
\]
\[
-\epsilon_{i_1}-\epsilon_{i_2} \mbox{ for }i_1>i_2\in I.
\]
Let $I=(i_1>i_2)$. If $i_1+i_2>0$, then the codimension of $\sigma'_I$ inside $\IGr(2,V)$ is the same as that of $\sigma_I$ inside $\I2Gr(2,V)$, and it is $\leq 2n-3$; if $i_1+i_2<0$, the codimension of $\sigma'_I$ inside $\IGr(2,V)$ is equal to $\codim(\sigma_I)+1\geq 2n-2$. Moreover $p_J\in \sigma_I$ implies that $p_J\in \sigma'_I$ and $\codim(\sigma'_J)>\codim(\sigma'_I)$. 

Therefore if $p_J\in \sigma_I$ and $\codim(\sigma_J)\leq \codim(\sigma_I)$, then the only possibility is that $\codim(\sigma_J)= \codim(\sigma_I)=2n-3$. As a consequence, $I$ must be of the form $I=(i,-i+1)$ for a certain $2\leq i\leq n$, and this forces either $J=(i-1,-i)$ or $J=(i,-i-1)$. 
\end{proof}

Moreover, some of the inclusions which hold in $\IGr(2,V)$ do not hold in $\I2Gr(2,V)$:

\begin{lemma}
\label{lemnotinclusions}
Let $I=(i,-i+1)$. If $J=(i-1,-i)$ or $J=(i-1,-i-1)$ or $J=(i-2,-i)$, then $p_J\notin \sigma_I$.
\end{lemma}

\begin{proof}
Let us fix some notation. We denote by $q_I$ the Pl\"ucker coordinates on the Grassmannian $\Gr(2,V)$. Then $\Gr(2,V)\subset \PP(\wedge^2 V)$ is defined by the quadratic equations
\begin{equation}
\label{eqpluck2}
q_{(a,b)}q_{(c,d)}-q_{(a,c)}q_{(b,d)}+q_{(b,c)}q_{(a,d)}=0 \mbox{ for }a,b,c,d\in\{\pm1,\dots,\pm n\}.
\end{equation}
Moreover the two equations defining the bisymplectic Grassmannian (and coming from $\omega_1$ and $\omega_2$) are:
\begin{equation}
\label{eqomegapluck2}
\sum_{i=1}^n q_{(i,-i)}=0 \mbox{ and }\sum_{i=1}^n \lambda_i q_{(i,-i)}=0.
\end{equation}
Finally, the Schubert variety $\sigma_I$ is defined by the relations
\[
q_{J}=0 \mbox{ for }I\ngeq J,
\]
while in a neighbourhood of $p_I$ we can suppose that $q_I\neq 0$. The relations defining the Schubert variety $\sigma_I$ and those coming from $\omega_1$ and $\omega_2$ imply that 
\[
q_{(i,-i)}=q_{(i-1,-i+1)}=0.
\]
By using the Plucker equations with $a=i,b=-i+1,c=i-1,d=-i-1$ (respectively $a=i,b=-i+1,c=i-1,d=-i$, $a=i,b=-i+1,c=i-2,d=-i$), one gets that $\sigma_I$ is contained in the locus where $q_{(i-1,-i-1)}=0$ (resp. $q_{(i-1,-i)}=0$, $q_{(i-2,-i)}=0$), which does not contain $p_{J}$ with $J=(i-1,-i-1)$ (resp. $J=(i-1,-i)$, $J=(i-2,-i)$).
\end{proof}

Now we are ready to prove the analogous of Proposition \ref{prophomvareqcohom}:

\begin{theorem}
\label{thmuniShckeqtwo}
The equivariant classes $f_I$ of Schubert varieties inside $\HHH_T^*(\I2Gr(2,V))$ are determined by the relations $1,2,3$ in Theorem \ref{thmlocgen}.
\end{theorem}

\begin{proof}
The polynomials $f_{I}(p_J)$ of the equivariant class of a Schubert variety $\sigma_I$ satisfy the relations in Theorem \ref{thmlocgen}. Moreover, by the finiteness of the number of $T$-invariant curves, we have that if two $T$-invariant curves with characters $\chi_1$, $\chi_2$ meet $p_I$, then $\chi_1$ and $\chi_2$ must be prime to each other.

Let us deal first with a Schubert variety $\sigma_I$, where $I$ is not of the form $I=(i,-i+1)$. This hypothesis implies by Lemma \ref{lemmacodimnullneg} that if $p_J\in \sigma_I$, then $\codim(\sigma_J)>\codim(\sigma_I)$. Let us consider an element 
\[
g=(g_1\dots,g_r)\in \HHH_T^*(X)\subset \CC[\Xi(T)]^{\oplus r}
\]
satisfying the relations $1,2,3$ in Theorem \ref{thmlocgen}. Then $f_{I}-g$ is zero over all points $p_J$ such that $\codim(\sigma_J)\leq \codim(\sigma_I)$. We want to prove that $f_{I}-g=0$. Let us suppose that $f_{I}-g\neq 0$. Then we can find a point $p_h\in\sigma_I$ such that $(f_{I}-g)(p_h)\neq 0$ and $\codim(\sigma_h)$ is minimal. Condition \eqref{reldeteqcohom} and the finiteness of the number of $T$-invariant curves implies that $(f_{I}-g)(p_h)$ must be divisible by $f_{h}(p_h)$ (because the weights of the tangent at any fixed point $p_L$ are exactly those of the $T$-equivariant curves linking the point to points $p_M$ with $M\geq L $); but 
\[
\deg((f_{I}-g)(p_h))=\codim(\sigma_J)<\codim(\sigma_h)=\deg(f_{h}(p_h)),
\] 
which gives a contradiction.

The previous argument must be adapted when $I=(i,-i+1)$. When this is the case, $\sigma_I$ can contain at most two points $p_h$ such that $\codim(\sigma_h)=\codim(\sigma_I)$, namely $h=(i,-i-1)$ and $h=(i-2,-i+1)$. Suppose for example that $h=(i,-i-1)$. In this case, as by Lemma \ref{lemnotinclusions} $p_{(i-1,-i-1)}\notin \sigma_I$, if $(f_{\sigma_I}-g)(p_h)\neq 0$ it must be divisible by $f_{\sigma_h}(p_h)(t_i-t_{i-1})$, whose degree is grater than the codimension of $\sigma_I$. Similarly when $h=(i-2,-i+1)$ because $p_{(i-2,-i)}\notin \sigma_I$.
\end{proof}

\subsubsection{An equivariant Lefschetz Hyperplane Theorem}

Let $i:\I2Gr(k,V)^T\to \I2Gr(k,V)$ be the inclusion of the fixed points, and $j:\I2Gr(k,V)\to \IGr(k,V)$ the natural inclusion. As $(i\circ j)^*:\HHH_T^*(\IGr(k,V))\to \HHH_T^*(\I2Gr(k,V)^T)$ is an inclusion (because $\I2Gr(k,V)^T=\IGr(k,V)^T$), we get that $j^*$ is injective as well. We will denote by $f_{\sigma'_I}(J)$ the pullback of the equivariant classes of Schubert subvarieties of $\IGr(k,V)$ inside $ \HHH_T^*(\I2Gr(k,V)^T)$. Moreover, let
\[
f_I f_J =\sum_{L}N_{I,J}^L f_L
\]
be the multiplication rule inside $\HHH_T^*(\I2Gr(k,V))$, and
\[
f_{\sigma_I} f_{\sigma_J} =\sum_{L}M_{I,J}^L f_{\sigma_L}
\]
the multiplication rule inside $\HHH_T^*(\IGr(k,V))$, where $I,J,L$ are admissible subsets and $N_{I,J}^L, M_{I,J}^L$ are polynomials of the right degree.

The Lefschetz hyperplane theorem says that the restriction of the cohomology of an ambient variety $X$ to an hypersurface $Y$ is an isomorphism in codimension $< \dim_{\CC}(X)$. The following proposition is an equivariant version of this classical result for bisymplectic Grassmannians of planes:

\begin{theorem}[Equivariant Lefschetz]
Let $I$ be an admissible subset of $\I2Gr(2,V)$ such that $\codim(\sigma_I)<2n-3$. Then 
\[
f_I=f_{\sigma'_I}.
\]
Moreover, let $J,L$ be two admissible subsets as well such that $\codim(\sigma_J)<2n-3$ and $\codim(\sigma_L)<2n-3$. Then 
\[
M_{I,J}^L=N_{I,J}^L.
\]
\end{theorem}

\begin{proof}
Let us consider $f_{\sigma'_I}$. By Lemma \ref{lemfinTcurvessym} $f_{\sigma'_I}$ satisfies the relations $1,2,3$ in Theorem \ref{thmlocgen} which are satisfied also by $f_I$, and all the relations in Theorem \ref{thmfinitecurvesgen}. Therefore, by Theorem \ref{thm1symgrass} and Theorem \ref{thmuniShckeqtwo} we get that $f_{\sigma'_I}=f_I$. The second statement follows at once.
\end{proof}

\begin{remark}
For what concerns the other classes, the problem becomes more involved. Indeed, if $\codim(\sigma_I)\geq 2n-3$, the class $f_{\sigma'_I}$ does not satisfy all the relations $1,2,3$ in Theorem \ref{thmlocgen}. For instance, if $\codim(\sigma_I)> 2n-3$, by the proof of Lemma \ref{lemmacodimnullneg} we know that $\codim(\sigma_I) =\codim(\sigma'_I)-1$, and therefore relation $1$ is not satisfied. Finding a formula which expresses all the classes $f_{\sigma'_I}$ in terms of the classes $f_I$ may help understanding better the equivariant cohomology of $\I2Gr(2,V)$. Indeed, one could try to derive an equivariant Pieri formula for multiplication of any Schubert class by a \emph{special} Schubert class, as it is done in \cite{Li2016} for the symplectic (as well as the ordinary and the orthogonal) Grassmannians.
\end{remark}

\subsection{A Chevalley formula}

The following lemma will be useful in the sequel:

\begin{lemma}
The Schubert variety $\sigma_H$, $H=\{ n,\dots, n-k+2,n-k\}$, corresponding to the unique generator of $\Pic(\I2Gr(k,V))$ is represented in equivariant cohomology by the degree $1$ polynomials
\[
f_H(I)=\sum_{i\in I} -\epsilon_i+\sum_{i=1}^k \epsilon_{n-i+1}.
\]
\end{lemma}

\begin{proof}
We already know that $f_H$ in the equivariant cohomology is uniquely determined by the fact that $f_H(\{ n,\dots, n-k+2,n-k+1\})=0$ and $f_H(H)=-\epsilon_{n-k}+\epsilon_{n-k+1}$. These conditions, together with condition \eqref{reldeteqcohom}, are satisfied by the formula in the statement.
\end{proof}

The next result we want to present is the computation of an equivariant Chevalley formula for bisymplectic Grassmannians of planes, i.e. of the coefficients $a_{I,J}$ appearing in Equation \eqref{eqindmethbisym} for $k=2$. Having these coefficients will permit to compute all the equivariant classes of Schubert varieties, starting from that of maximal codimension up to the one of codimension $0$. We will divide the proof in different lemmas, which deal with different situations. The most difficult part will be understanding the behaviour of classes of middle codimension, because in this case we have coefficients $a_{I,J}$ which are of degree one, and not just constants (Lemma \ref{lemmacodimnullneg}). At the end of the proof we have summarized the Chevalley formula in Theorem \ref{thmChevform}.

The first lemma deals with Schubert varieties for which the Chevalley formula is the same as that of symplectic Grassmannians:

\begin{lemma}
\label{lemnormchev}
Let $I,J$ be admissible subsets such that either $\codim(\sigma_J)<2n-3$ or $\codim(\sigma_I)>2n-3$. If $\#(I\cap J)=1$ and $\codim(\sigma_I)=\codim(\sigma_J)-1$ then $a_{I,J}=1$, otherwise $a_{I,J}=0$.
\end{lemma}

\begin{proof}
By hypothesis, we have that $\codim(\sigma_I)-\codim(\sigma_J)=\codim(\sigma'_I)-\codim(\sigma'_J)$. Therefore $p_J\in \sigma_I$ and $\codim(\sigma_I)=\codim(\sigma_J)-1$ implies that $\#(I\cap J)=1$. Let us suppose that $I=(i_1,i_2)$ and $J=(j_1,j_2)$ are admissible subsets, and that $i_1=j_1$ (the case $i_2=j_2$ is treated similarly). By Equation \eqref{eqindmethbisym} and by the fact that there exists a $T$-invariant curve between $p_I$ and $p_J$ of weight $t_{i_2}-t_{j_2}$, we have the two following relations:
\[
f_I(J)(t_{i_2}-t_{j_2})=a_{I,J}f_J(J),
\]
\[
f_I(I)-f_I(J) \mbox{ is divisible by }(t_{i_2}-t_{j_2}).
\]
Putting them together, as $f_I(I)$ and $f_J(J) / (t_{i_2}-t_{j_2})$ are not divisible by $(t_{i_2}-t_{j_2})$, we have that 
\[
f_I(I)-a_{I,J}\frac{f_J(J)}{t_{i_2}-t_{j_2}}\equiv 0 \mbox{ $\modulo$ } (t_{i_2}-t_{j_2})
\]
implies that $a_{I,J}=1$ because this ensures that the LHS is equal to zero (even not modulo $(t_{i_2}-t_{j_2})$). In the other cases when $\codim(\sigma_I)=\codim(\sigma_J)-1$, the coefficient $a_{I,J}=0$ by applying Equation \eqref{eqindmethbisym} to $J$ because $p_J\notin p_I$.
\end{proof}

\begin{figure}
 \caption{Inclusions of fixed points and $T$-invariant curves inside $\I2Gr(2,V)$ in codimension $=2n-3$}
  \label{Figure-1}
\begin{center}
 \begin{tikzpicture}
    \tikzstyle{every node}=[draw,circle,fill=white,minimum size=4pt,
                            inner sep=0pt]

      
      \draw (-6,0) node (1) [label=above: (n\comma -n+1)] {};
      \draw (-4.5,0) node (2) [label=below: (n-1\comma -n+2)] {};
      \draw (-3,0) node (3) [label=above: (n-2\comma -n+3)] {};
      \draw (-1.5,0) node (4) [label=below: (n-3\comma -n+4)] {};
      \draw (-0.5,0) node (0) {};
      \draw (-0.15,0) node (00) {};
      \draw (0.15,0) node (000) {};
      \draw (0.5,0) node (0000) {};
      \draw (1.5,0) node (5) [label=above: (n-4\comma -n+3)] {};
      \draw (3,0) node (6) [label=below: (n-3\comma -n+2)] {};
      \draw (4.5,0) node (7) [label=above: (n-2\comma -n+1)] {};
      \draw (6,0) node (8) [label=below: (n-1\comma -n)] {};

    \draw (1) edge [out=30, in=150] (7);
    \draw (3) edge [out=-25, in=-155] (7);
    \draw (3) edge [out=20, in=160] (5);
    \draw (0) edge [out=-20, in=-160] (5);
    \draw (2) edge [out=-30, in=-150] (8);
    \draw (2) edge [out=25, in=155] (6);
    \draw (4) edge [out=-20, in=-160] (6);
    \draw (4) edge [out=20, in=160] (0000);

\end{tikzpicture}
\end{center}
\end{figure}
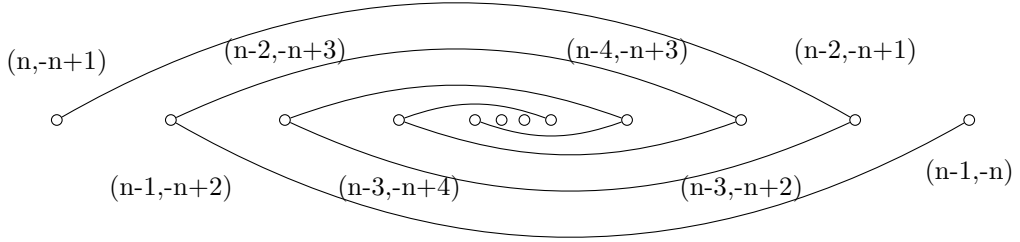

The following lemmas deal with the interesting part of the cohomology of $\I2Gr(2,V)$. In Figure \ref{Figure-1} we reported the inclusions of fixed points whose Schubert varieties have the same codimension (situation described in Lemma \ref{lemmacodimnullneg} and Lemma \ref{lemnotinclusions}). We deal first of all with these inclusions, i.e. with polynomials $a_{I,J}$ of degree $1$:

\begin{lemma}
\label{lemChevnotconst}
If $I=(i,-i+1)$ and $J=(i,-i-1)$ or $J=(i-2,-i+1)$ with $i>0$, then $a_{I,J}=t_{i-1}-t_{i}$.
\end{lemma}

\begin{proof}
Let us suppose $J=(i,-i-1)$. By Lemma \ref{lemnotinclusions}, we know that $p_{(i-1,-i-1)}\notin \sigma_I$; therefore $f_I((i-1,-i-1))=0$ and the existence of a $T$-invariant curve between $p_J$ and $p_{(i-1,-i-1)}$ gives that 
\[
f_I(J) \mbox{ is divisible by }(t_{i-1}-t_{i}).
\]
As by Equation \eqref{eqindmethbisym}
\[
f_I(J)(t_{i+1}-t_{i-1})=a_{I,J}f_J(J),
\]
and as by Theorem \ref{thmlocgen} $f_J(J)$ is not divisible by $(t_{i-1}-t_{i})$, we get that the coefficient $a_{I,J}$ is of the form $a_{I,J}=a(t_{i-1}-t_{i})$, for a certain constant $a$. We have that $a=1$ by the existence of a $T$-invariant curve between $p_I$ and $p_J$, which gives the relation determining $a$:
\[
f_I(I)-a\frac{t_{i-1}-t_{i}}{t_{i+1}-t_{i-1}}f_J(J)\equiv 0\mbox{ $\modulo$ }(t_{i+1}-t_{i-1}).
\]
If $J=(i-2,-i+1)$, the proof is exactly the same, provided that we replace $p_{(i-1,-i-1)}\notin \sigma_I$ by the fact that $p_{(i-2,-i)}\notin \sigma_I$.
\end{proof}

The following facts can be verified easily: if $\codim(\sigma_I)=2n-2$, then either $I=(i,-i+2)$ with $i>0$ or $I=(2,1)$. By symmetry, if $\codim(\sigma_I)=2n-4$, then $I=(i-2,-i)$ with $i>0$ or $I=(-1,-2)$. Finally, if $\codim(\sigma_I)=2n-3$, then either $I=(i,-i+1)$ or $I=(i-1,-i)$ with $i>0$.

Figure \ref{Figure-3} represents the inclusions of fixed points which are relevant for the following proposition:

\begin{lemma}
Let $I=(i-1,-i)$, with $i>0$. The only non-zero coefficients $a_{I,J}$ are:
\[
a_{I,(i-2,-i)}=a_{I,(i-1,-i-1)}=1.
\]
\end{lemma}

\begin{proof}
The proof of this result follows the same lines of the proof of Lemma \ref{lemnormchev}. The reason why this happens is that in this case as well $\codim(\sigma_I)-\codim(\sigma_J)=\codim(\sigma'_I)-\codim(\sigma'_J)$.
\end{proof}

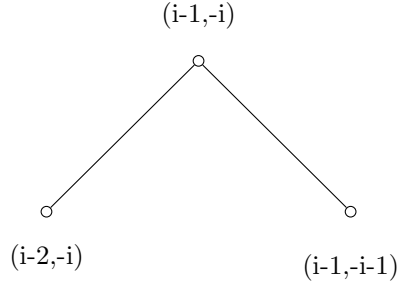
\begin{figure}
 \caption{Inclusions of fixed points inside $\sigma_I$ with $I=(i-1,-i)$, $i>0$}
  \label{Figure-3}
\begin{center}
 \begin{tikzpicture}
    \tikzstyle{every node}=[draw,circle,fill=white,minimum size=4pt,
                            inner sep=0pt]

      \draw (0,2) node (1) [label=above: (i-1\comma -i)] {};
      \draw (-2,0) node (2) [label=below: (i-2\comma -i)] {};
      \draw (2,0) node (3) [label=below: (i-1\comma -i-1)] {};

    \draw (1) -- (2);
    \draw (1) -- (3);

\end{tikzpicture}
\end{center}
\end{figure}

Figure \ref{Figure-2} represents the inclusions of fixed points which are relevant for the following proposition:

\begin{lemma}
\label{lemmaprova}
Let $I=(i,-i+2)$, with $i>0$. The only non-zero coefficients $a_{I,J}$ are:
\[
a_{I,(i,-i+1)}=a_{I,(i-1,-i+2)}=a_{I,(i-3,-i+2)}=a_{I,(i,-i-1)}=1,
\]
\[
a_{I,(i-2,-i+1)}=a_{I,(i-1,-i)}=2.
\]
\end{lemma}

\begin{figure}
 \caption{Inclusions of fixed points inside $\sigma_I$ with $I=(i,-i+2)$, $i>0$}
  \label{Figure-2}
\begin{center}
 \begin{tikzpicture}
    \tikzstyle{every node}=[draw,circle,fill=white,minimum size=4pt,
                            inner sep=0pt]

      \draw (0,2) node (1) [label=above: (i\comma -i+2)] {};
      \draw (-4,0) node (2) [label=below: (i\comma -i+1)] {};
      \draw (-2,0) node (3) [label=below: (i-1\comma -i+2)] {};
      \draw (2,0) node (4) [label=below: (i-3\comma -i+2)] {};
      \draw (4,0) node (5) [label=below: (i-2\comma -i+1)] {};
      \draw (6,0) node (6) [label=below: (i-1\comma -i)] {};
      \draw (8,0) node (7) [label=below: (i\comma -i-1)] {};

    \draw (1) -- (2);
    \draw (1) -- (3);
    \draw (1) -- (4);
    \draw (1) -- (5);
    \draw (1) -- (6);
    \draw (1) -- (7);

\end{tikzpicture}
\end{center}
\end{figure}
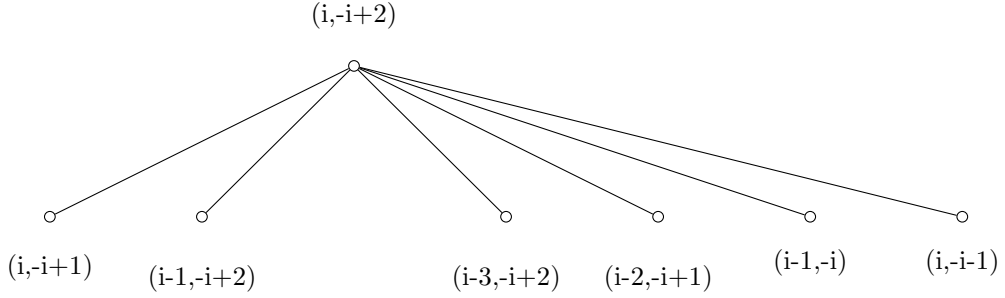

\begin{proof}
The coefficients $a_{I,(i,-i+1)}$ and $a_{I,(i-1,-i+2)}$ are computed as it is done in the proof of Lemma \ref{lemnormchev}. Let us deal with the remaining coefficients for points $p_J\in \sigma_I$ such that $\codim(\sigma_I)=\codim(\sigma_J)-1$:
\begin{itemize}
\item Let $J=(i-3,-i+2)$. By Equation \eqref{eqindmethbisym}, 
\[
f_I(J)(t_i-t_{i-3})=f_{(i-1,-i+2)}(J)+a_{I,J}f_J(J).
\]
Moreover by Lemma \ref{lemChevnotconst} we know that 
\[
f_{(i-1,-i+2)}(J)=\frac{t_{i-2}-t_{i-1}}{t_{i-1}-t_{i-3}}f_J(J).
\]
Therefore, the existence of a $T$-equivariant curve between $p_I$ and $p_J$ of weight $(t_i-t_{i-3})$ gives the relation 
\[
f_I(I)-\frac{t_{i-2}-t_{i-1}+a_{I,J}(t_{i-1}-t_{i-3})}{(t_i-t_{i-3})(t_{i-1}-t_{i-3})}f_J(J)\equiv 0 \mbox{ $\modulo$ }(t_i-t_{i-3}).
\]
As $f_I(I)$ is divisible by $(t_{i-2}-t_{i-3})$ and not by $f_J(J)$, we get that $a_{I,J}=1$.
\item Let $J=(i,-i-1)$. The argument is similar to the previous one; the last relation becomes
\[
f_I(I)-\frac{t_{i-1}-t_{i}+a_{I,J}(t_{i+1}-t_{i-1})}{(t_{i+1}-t_{i-2})(t_{i+1}-t_{i-1})}f_J(J)\equiv 0 \mbox{ $\modulo$ }(t_{i+1}-t_{i-2}).
\]
As $f_I(I)$ is divisible by $(t_{i+1}-t_{i})$, we get that $a_{I,J}=1$.
\item Let $J=(i-2,-i+1)$. Lemma \ref{lemChevnotconst} gives
\[
f_{(i,-i+1)}(J)=(t_{i-1}-t_i)f_J(J).
\]
Using this relation and Equation \eqref{eqindmethbisym} we obtain
\[
f_I(J)(t_i-2t_{i-2}+t_{i-1})=\frac{t_{i-1}-t_i+a_{I,J}(t_i-t_{i-2})}{t_i-t_{i-2}}f_J(J),
\]
which implies that $a_{I,J}=2$.
\item Let $J=(i-1,-i)$. The argument is similar to the previous one; Lemma \ref{lemChevnotconst} and Equation \eqref{eqindmethbisym} give the relation
\[
f_I(J)(2t_i-t_{i-2}-t_{i-1})=\frac{t_{i-2}-t_{i-1}+a_{I,J}(t_i-t_{i-2})}{t_i-t_{i-2}}f_J(J),
\]
which implies that $a_{I,J}=2$.
\end{itemize}
\end{proof}

Figure \ref{Figure-4} represents the inclusions of fixed points which are relevant for the following proposition:

\begin{lemma}
Let $I=(i,-i+1)$, with $i>0$. The only non-zero constant coefficients $a_{I,J}$ are:
\[
a_{I,(i-1,-i-1)}=a_{I,(i-2,-i)}=1,
\]
\[
a_{I,(i,-i-2)}=a_{I,(i-3,-i+1)}=0.
\]
\end{lemma}

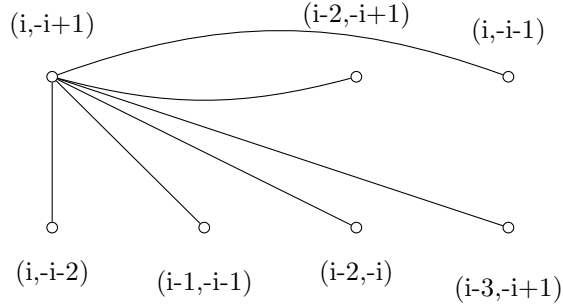
\begin{figure}
 \caption{Inclusions of fixed points inside $\sigma_I$ with $I=(i,-i+i)$, $i>0$}
  \label{Figure-4}
\begin{center}
 \begin{tikzpicture}
    \tikzstyle{every node}=[draw,circle,fill=white,minimum size=4pt,
                            inner sep=0pt]

      \draw (-2,2) node (1) [label=above: (i\comma -i+1)] {};
      \draw (2,2) node (2) [label=above: (i-2\comma -i+1)] {};
      \draw (4,2) node (3) [label=above: (i\comma -i-1)] {};
      \draw (-2,0) node (4) [label=below: (i\comma -i-2)] {};
      \draw (0,0) node (5) [label=below: (i-1\comma -i-1)] {};
      \draw (2,0) node (6) [label=below: (i-2\comma -i)] {};
      \draw (4,0) node (7) [label=below: (i-3\comma -i+1)] {};

    \draw (1) edge [out=-15, in=-165] (2);
    \draw (1) edge [out=20, in=160] (3);
    \draw (1) -- (4);
    \draw (1) -- (5);
    \draw (1) -- (6);
    \draw (1) -- (7);

\end{tikzpicture}
\end{center}
\end{figure}

\begin{proof}
The proof uses the same arguments of the proof of Lemma \ref{lemmaprova}, therefore we will be more concise. We need to deal with the coefficients for points $p_J\in \sigma_I$ such that $\codim(\sigma_I)=\codim(\sigma_J)-1$:
\begin{itemize}
\item Let $J=(i-1,-i-1)$. Lemma \ref{lemChevnotconst} and Equation \eqref{eqindmethbisym} give the relation
\[
f_I(J)(t_i-2t_{i-2}+t_{i-1})=\frac{t_{i-2}-t_{i-1}+a_{I,J}(t_{i-1}-t_{i-2})}{t_{i-1}-t_{i-2}}f_J(J),
\]
which implies that $a_{I,J}=1$ because by Lemma \ref{lemnotinclusions} $f_I(J)=0$.
\item Let $J=(i-2,-1)$. Lemma \ref{lemChevnotconst} and Equation \eqref{eqindmethbisym} give the relation
\[
f_I(J)(2t_i-t_{i-2}-t_{i-1})=\frac{t_{i-1}-t_{i}+a_{I,J}(t_{i}-t_{i-1})}{t_{i}-t_{i-1}}f_J(J),
\]
which implies that $a_{I,J}=1$ because by Lemma \ref{lemnotinclusions} $f_I(J)=0$.
\item Let $J=(i,-i-2)$. By using Lemma \ref{lemChevnotconst} and Equation \eqref{eqindmethbisym} repeatedly, and the existence of a $T$-equivariant curve between $p_I$ and $p_J$, we obtain the relation
\[
f_I(I)-\frac{t_{i-1}-t_{i}+a_{I,J}(t_{i+2}-t_{i+1})}{(t_{i+2}-t_{i-1})(t_{i+2}-t_{i+1})}f_J(J)\equiv 0 \mbox{ $\modulo$ }(t_{i+2}-t_{i-1}).
\]
As $f_I(I)$ is divisible by $(t_{i+2}-t_{i})$, we get that $a_{I,J}=0$.
\item Let $J=(i-3,-i+1)$. By using Lemma \ref{lemChevnotconst} and Equation \eqref{eqindmethbisym} repeatedly, and the existence of a $T$-equivariant curve between $p_I$ and $p_J$, we obtain the relation
\[
f_I(I)-\frac{t_{i-1}-t_{i}+a_{I,J}(t_{i-2}-t_{i-3})}{(t_{i}-t_{i-3})(t_{i-2}-t_{i-3})}f_J(J)\equiv 0 \mbox{ $\modulo$ }(t_{i}-t_{i-3}).
\]
As $f_I(I)$ is divisible by $(t_{i-1}-t_{i-3})$, we get that $a_{I,J}=0$.
\end{itemize}
\end{proof}

Putting all the lemmas together, we have proved:

\begin{theorem}[Chevalley formula]
\label{thmChevform}
The coefficients $a_{I,J}$ for $I=(i_1,i_2),J=(j_1,j_2)$ two admissible subsets in the Chevalley formula \eqref{eqindmethbisym} for the bisymplectic Grassmannian of planes $\I2Gr(2,V)$ are given by the following rules (the integer $i$ is always supposed to be $>0$):
\[
a_{(i,-i+1),(i,-i-1)}=a_{(i,-i+1),(i-2,-i+1)}=t_{i-1}-t_i;
\]
\[
a_{(i,-i+2),(i-3,-i+2)}=a_{(i,-i+2),(i,-i-1)}=a_{(i,-i+1),(i-1,-i-1)}=a_{(i,-i+1),(i-2,-i)}=1;
\]
\[
a_{(i,-i+2),(i-2,-i+1)}=a_{(i,-i+2),(i-1,-i)}=2;
\]
in all the other cases either $I\geq J$, $\#(I\cap J)=1$, $ \codim(\sigma_I)=\codim(\sigma_J)-1$ and $a_{I,J}=1$, or $a_{I,J}=0$.
\end{theorem}  

Thus, we obtain:

\begin{corollary}
\label{thm1coho22n}
Equation \eqref{eqindmethbisym} and Theorem \ref{thmChevform} determine inductively the equivariant classes of all the Schubert varieties inside $\I2Gr(2,V)$.
\end{corollary}

\begin{remark}[$\I2Gr(2,8)$]
Let us point out that the constant coefficients $a_{I,J}$ computed in Theorem \ref{thmChevform} give the Chevalley formula for the classical cohomology (by Theorem \ref{thmeqclasscohomgen}), and therefore allow to compute the degrees of Shubert varieties. In Figure \ref{Figure2,8} we reported the degrees of Schubert varieties inside $\I2Gr(2,8)$ (the case of $\I2Gr(2,6)$ will be dealt with in the next section). As it was expected classically, we find that the degree of $\I2Gr(2,8)$ is equal to $\deg(\Gr(2,8))=132$, and this is an evidence of the fact that our formula is correct.
\end{remark}

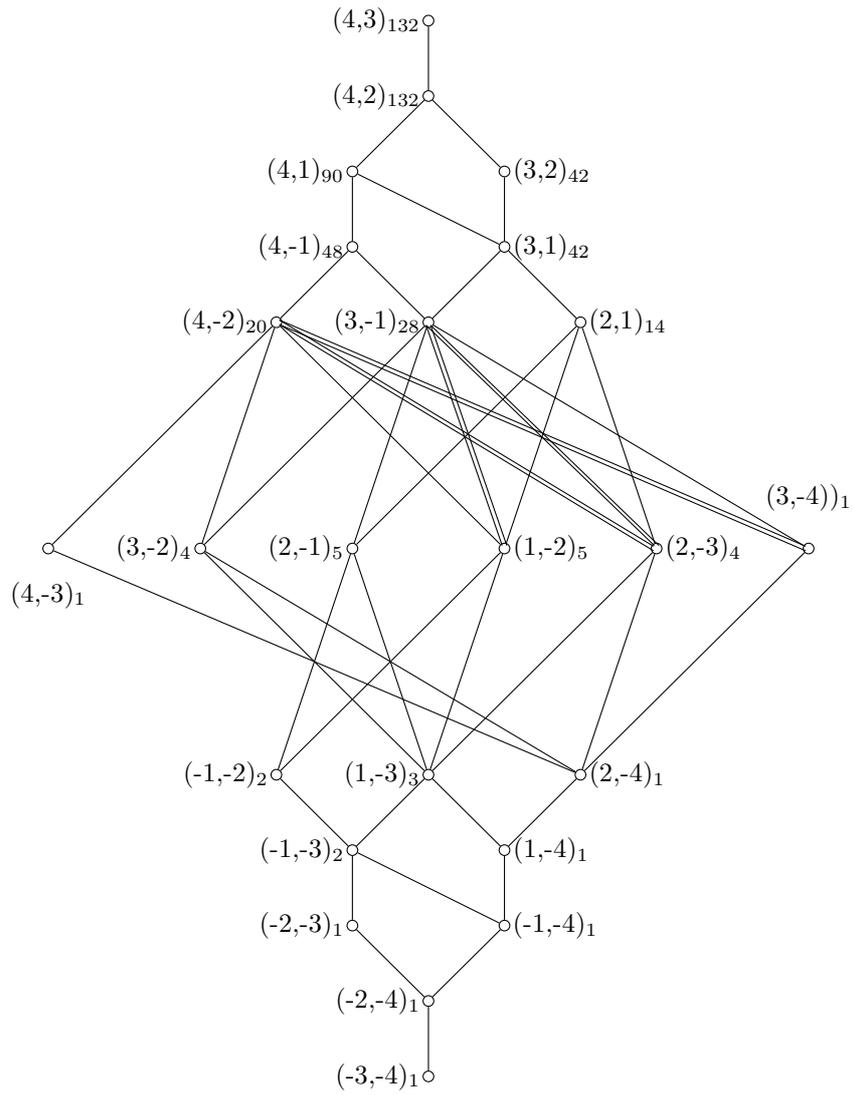
\begin{figure}
 \caption{Degree of Schubert varieties (notation $I_{\deg(\sigma_I)}$) inside $\I2Gr(2,8)$}
  \label{Figure2,8}
\begin{center}
 \begin{tikzpicture}
    \tikzstyle{every node}=[draw,circle,fill=white,minimum size=4pt,
                            inner sep=0pt]

      \draw (0,7) node (43) [label=left: (4\comma 3$)_{132}$] {};
      \draw (0,6) node (42) [label=left: (4\comma 2$)_{132}$] {};
      \draw (-1,5) node (41) [label=left: (4\comma 1$)_{90}$] {};
      \draw (1,5) node (32) [label=right: (3\comma 2$)_{42}$] {};
      \draw (-1,4) node (4-1) [label=left: (4\comma -1$)_{48}$] {};
      \draw (1,4) node (31) [label=right: (3\comma 1$)_{42}$] {};
      \draw (-2,3) node (4-2) [label=left: (4\comma -2$)_{20}$] {};
      \draw (0,3) node (3-1) [label=left: (3\comma -1$)_{28}$] {};
      \draw (2,3) node (21) [label=right: (2\comma 1$)_{14}$] {};
      
      \draw (-5,0) node (4-3) [label=below: (4\comma -3$)_1$] {};
      \draw (-3,0) node (3-2) [label=left: (3\comma -2$)_4$] {};
      \draw (-1,0) node (2-1) [label=left: (2\comma -1$)_5$] {};
      \draw (1,0) node (1-2) [label=right: (1\comma -2$)_5$] {};
      \draw (3,0) node (2-3) [label=right: (2\comma -3$)_4$] {};
      \draw (5,0) node (3-4) [label=above: (3\comma -4)$)_1$] {};
      
      \draw (0,-7) node (-3-4) [label=left: (-3\comma -4$)_1$] {};
      \draw (0,-6) node (-2-4) [label=left: (-2\comma -4$)_1$] {};
      \draw (-1,-5) node (-2-3) [label=left: (-2\comma -3$)_1$] {};
      \draw (1,-5) node (-1-4) [label=right: (-1\comma -4$)_1$] {};
      \draw (-1,-4) node (-1-3) [label=left: (-1\comma -3$)_2$] {};
      \draw (1,-4) node (1-4) [label=right: (1\comma -4$)_1$] {};
      \draw (-2,-3) node (-1-2) [label=left: (-1\comma -2$)_2$] {};
      \draw (0,-3) node (1-3) [label=left: (1\comma -3$)_3$] {};
      \draw (2,-3) node (2-4) [label=right: (2\comma -4$)_1$] {};

      \draw (43) -- (42);
      \draw (42) -- (41);
      \draw (42) -- (32);
      \draw (41) -- (4-1);
      \draw (41) -- (31);
      \draw (32) -- (31);
      \draw (4-1) -- (4-2);
      \draw (4-1) -- (3-1);
      \draw (31) -- (3-1);
      \draw (31) -- (21);
      
      \draw (-3-4) -- (-2-4);
      \draw (-2-4) -- (-1-4);
      \draw (-2-4) -- (-2-3);
      \draw (-1-4) -- (1-4);
      \draw (-1-4) -- (-1-3);
      \draw (-2-3) -- (-1-3);
      \draw (1-4) -- (2-4);
      \draw (1-4) -- (1-3);
      \draw (-1-3) -- (1-3);
      \draw (-1-3) -- (-1-2);
      
      \draw (4-2) -- (4-3);
      \draw (4-2) -- (3-2);
      \draw (4-2) -- (1-2);   
      \draw (3-1) -- (3-2);
      \draw (3-1) -- (2-1);
      \draw (3-1) -- (3-4);
      \draw (21) -- (2-1);
      \draw (21) -- (1-2);
      \draw (21) -- (2-3);
      
      \draw (2-1) -- (-1-2);
      \draw (1-2) -- (-1-2);
      \draw (3-2) -- (1-3);
      \draw (2-1) -- (1-3);
      \draw (1-2) -- (1-3);
      \draw (2-3) -- (1-3);
      \draw (4-3) -- (2-4);
      \draw (3-2) -- (2-4);
      \draw (2-3) -- (2-4);
      \draw (3-4) -- (2-4);
      
    \draw (-1.96,3.03) -- (2.96,0.03);
    \draw (-1.96,2.97) -- (2.96,-0.03);
    \draw (-1.96,3.03) -- (4.96,0.03);
    \draw (-1.96,2.97) -- (4.96,-0.03);
    
    \draw (-0.03,2.96) -- (2.97,0.04);
    \draw (0.03,2.96) -- (3.03,0.04);
    \draw (-0.03,2.96) -- (0.97,0.04);
    \draw (0.03,2.96) -- (1.03,0.04);


\end{tikzpicture}
\end{center}
\end{figure}

\subsection{A quasi-homogeneous example}
\label{i2gr26}

As an application of the previous general results, in this section we study in detail the smallest non-trivial bisymplectic Grassmannian of planes, i.e. $\I2Gr(2,V)$ with $V\cong \CC^6$. This variety is interesting not only because computations are still feasible by hand, but because it is a quasi-homogeneous variety, i.e. it admits an action of a group with a dense orbit. Moreover, it has no small deformations, and it admits only a finite number of flat deformations. In the following we study its decomposition in orbits and its flat deformations. Then, we will give a presentation of its (classical) cohomology ring.

The variety $\I2Gr(2,V)$ with $V\cong \CC^6$ admits an action of 
\[
G=\SL(2)^3\cong \SL(K_1)\times \SL(K_2)\times \SL(K_3),
\]
where the $2$-dimensional planes $K_1,K_2,K_3$ have been defined in Section \ref{secsmooth}. We will denote a vector inside $K_i$ by the subscript $i$ (e.g. $v_i,v'_i$, etc.). The list of $G$-orbits inside $\I2Gr(2,V)$ with their representatives is the following one:
\begin{itemize}
\item A representative of the dense orbit is $[P]=(v_1+v_2+v_3)\wedge(v'_1+v'_2+v'_3)$. This orbit is isomorphic to the quotient $\SL(2)^3 / \SL(2)$, where the quotient factor $\SL(2)$ is the image of the diagonal morphism $\SL(2)\to SL(2)^3$. Being the quotient of two reductive groups, the dense orbit is an affine variety. Indeed, as the Plucker coordinate $q_{(1,-1)}([P])\neq 0$, all the points $[Q]$ of the orbit satisfy $q_{(1,-1)}([Q])\neq 0$. Therefore the orbit is contained inside the affine variety $\{q_{(1,-1)}\neq 0\}\subset \I2Gr(2,V)$; in fact the dense orbit is equal to $\{q_{(1,-1)}\neq 0\}$ (or equivalently $q_{(2,-2)}\neq 0$ or $q_{(3,-3)}\neq 0$).
\item There is one orbit with representatives of type $(v_i+v_j)\wedge (v_j+v_k)$ (or, which is the same, $(v_i+2v_j+v_k)\wedge (v_i+v_j)$). Let $\cU_i$ be the tautological bundle over $\PP(K_i)$. Then this orbit is isomorphic to the total space of
\[
(\PP(\cU_i\oplus \cU_j)\setminus (\PP(\cU_i)\cup \PP(\cU_j)))\times (\PP(\cU_k\oplus \cU_j)\setminus (\PP(\cU_k)\cup \PP(\cU_j)))
\]
\[
\mbox{ over }\PP(K_1)\times \PP(K_2)\times\PP(K_3).
\]
Its closure is the irreducible divisor that compactifies the dense orbit.
\item There are three orbits with representatives of type $v_i\wedge (v_j+v_k)$, each one isomorphic to 
\[
\PP(K_i)\times(\PP(K_j\oplus K_k)\setminus (\PP(K_j)\cup \PP(K_k))).
\]
\item There are three minimal orbits with representatives of type $v_i\wedge v_j$, each one isomorphic to 
\[
\PP(K_i)\times \PP(K_j).
\]
\end{itemize}

\subsubsection{The Hilbert scheme of $\I2Gr(2,6)$}

We have already seen that $\I2Gr(2,6)$ has no small deformations (Theorem \ref{thmsmalldefbisym}), and that there is only one smooth isomorphism class (see Remark \ref{thmsmalldefbisym}). This is related to the fact that if $V\cong \CC^6$, then $ (\wedge^2 V^*)\otimes \CC^2$ is a prehomogeneous space for the action of $\SL(V)\times \SL(2)\times \CC^*$ (see \cite{WeymanE7}). This implies that there are just a finite number of orbits, and therefore that all pencils $\Omega$ in a dense subset of $\PP(\wedge^2 V^*)$ are conjugated under the action of $\PGL(V)$. As a consequence, there are only finitely many isomorphism classes of varieties of the form $\zero(\Omega)$. In the following we intend to describe these varieties.

We will consider $\I2Gr(2,V)$ with $V\cong \CC^6$ as a subvariety $\I2Gr(2,V)\subset\Gr(2,V)\subset \PP(\wedge^2 V^*)$, and we will denote by $p(t)$ the Hilbert polynomial 
\[
p(t)=\chi(\I2Gr(2,V), \cO(t)) =\HHH^0(\I2Gr(2,V), \cO(t)) \mbox{ for }t>>0.
\]

\begin{proposition}
\label{HilbI2Gr26}
There are $11$ flat deformations (included the smooth one) of $\I2Gr(2,V)$ inside $\Gr(2,V)$. They correspond to the orbits of $\SL(V)$ inside $\Gr(2,\wedge^2 V^*)$, which can be identified as a smooth component of the Hilbert scheme of $\I2Gr(2,V)\subset \Gr(2,V)$.
\end{proposition}

\begin{proof}
Let us consider a pencil $\Omega$. In order to have that $\zero(\Omega)$ is a flat deformation of the (smooth) bisymplectic Grassmannian, we only need to verify that it has the expected codimension (equal to $6$). Indeed, in that case, we can compute $p(t)=\chi(\I2Gr(2,V), \cO(t))$ by using the Koszul complex as
\[
p(t)=\chi(Gr(2,V), \cO(t))-\chi(\Gr(2,V), 2\cO(t-1))+\chi(\Gr(2,V), \cO(t-2)),
\] 
obtaining that the Hilbert polynomial does not depend on the particular choice of $\Omega$. 

By \cite{WeymanE7}[Case $E_7,\alpha_3$], there are $15$ orbits of $\SL(V)\times \SL(2)\times \CC^*$ inside $(\wedge^2 V^*)\otimes \CC^2$. Four of them are generated by one form, therefore the corresponding zero locus $\zero(\Omega)$ has dimension $\geq 7$ and cannot be a flat deformation of the (smooth) bisymplectic Grassmannian. The orbits of actual pencils $\Omega$ have been reported in Figure \ref{Figure0}. Among them:
\begin{itemize}
\item the pencils inside $O_{0},O_{1},O_{2},O_{5_I},O_{6}$ contain a non-degenerate form, therefore $\zero(\Omega)$ is a hypersurface in the irreducible variety $\IGr(2,V)$ and has dimension equal to $6$;
\item the pencils inside $O_{7},O_{10},O_{11},O_{15}$ contain a form of type $x_1\wedge x_{-1}$, whose zero locus defines a (irreducible) Schubert variety inside $\Gr(2,V)$. Therefore $\zero(\Omega)$ is again $6$-dimensional;
\item the pencils inside $O_{4},O_{5_{II}}$ contain a form of type $x_1\wedge x_{-1}+x_{2}\wedge x_{-2}$, which is singular only at one point and irreducible as well. Therefore once more $\zero(\Omega)$ is $6$-dimensional.
\end{itemize}

We have thus shown that the family $\{(\zero(\Omega),\Omega)\subset \Gr(2,V)\times \Gr(2,\wedge^2 V^*)\}$ is flat over $\Gr(2,\wedge^2 V^*)$, and this gives a morphism $\psi$ from $\Gr(2,\wedge^2 V^*)$ to the Hilbert scheme of $\I2Gr(2,V)\subset \Gr(2,V)$. Moreover, this Hilbert scheme has tangent space at $\psi(\Omega)=\zero(\Omega)$ equal to
\[
\HHH^0(\zero(\Omega), \mathcal{N}_{\zero(\Omega), \Gr(2,V)})=\HHH^0(\zero(\Omega), 2\cO(1))\cong T_{\Gr(2,\wedge^2 V^*),\Omega},
\]
and the differential of the morphism $\psi$ is an isomorphism at each point (notice that the chain of isomorphisms does not depend on the fact that $\zero(\Omega)$ is smooth). We get that $\psi$ is \'etale; moreover, it is injective because $\Omega$ can be recovered as the codimension two linear space inside $\wedge^2 V$ generated by the linear system $|\cO(1)|$ over $\zero(\Omega)$. Therefore $\Gr(2,\wedge^2 V^*)$ is exactly one irreducible component of the Hilbert scheme.
\end{proof}

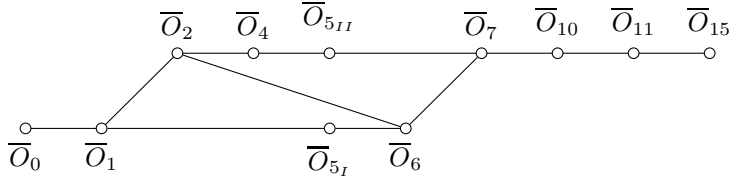
\begin{figure}
 \caption{Orbit closures of non-degenerate pencils of $2$-forms with respective codimensions as labels}
  \label{Figure0}
\begin{center}
 \begin{tikzpicture}
    \tikzstyle{every node}=[draw,circle,fill=white,minimum size=4pt,
                            inner sep=0pt]

    \draw (-5,0) node (14) [label=below:$\overline{O}_{0}$] {}
       -- ++(0:1cm) node (13) [label=below:$\overline{O}_{1}$] {}
        -- ++(45:1.41cm) node (11) [label=above:$\overline{O}_{2}$] {}
         -- ++(0:1cm) node (10) [label=above:$\overline{O}_{4}$] {}
        -- ++(0:1cm) node (8) [label=above:$\overline{O}_{5_{II}}$] {}
        -- ++(0:2cm) node (7) [label=above:$\overline{O}_{7}$] {}
        -- ++(0:1cm) node (6) [label=above:$\overline{O}_{10}$] {}
        -- ++(0:1cm) node (5) [label=above:$\overline{O}_{11}$] {}
        -- ++(0:1cm) node (3) [label=above:$\overline{O}_{15}$] {};

    \draw (13) -- ++(0:3cm) node (12) [label=below:$\overline{O}_{5_I}$] {};
   \draw (12) -- ++(0:1cm) node (9) [label=below:$\overline{O}_{6}$] {};
    \draw (11) -- (9);
    \draw (9) -- (7);

\end{tikzpicture}
\end{center}
\end{figure}

\subsubsection{Presentation of the cohomology for $\I2Gr(2,6)$}

In this last section, we compute explicitly the (equivariant) cohomology of $\I2Gr(2,V)$ for $V\cong \CC^6$. We give a presentation of the cohomology ring and we discuss some related questions, such as the existence of a certain symmetry or of a self-dual basis. We begin with an application of the Chevalley formula for bisymplectic Grassmannians of planes:

\begin{proposition}
\label{coeffaibisym}
The coefficients $a_{I,J}$ that appear in Equation \eqref{eqindmethbisym} for $\I2Gr(2,6)$ are uniquely determined by the relations in Theorem \ref{thmlocgen}. They are reported in Figure \ref{Figure3}.
\end{proposition}

Thus, by Corollary \ref{thm1coho22n}, we know that it is possible to determine inductively the equivariant classes of all the Schubert varieties inside $\I2Gr(2,6)$.

\begin{remark}
The constant coefficients $a_{I,J}$ determine the multiplication of a Schubert variety with the hyperplane section in the ordinary cohomology, i.e. a Pieri type formula for $\I2Gr(2,6)$. In particular, our computations are coherent with the fact that the degree of $\I2Gr(2,6)$ is $14$, as we know because it is the degree of $\Gr(2,6)$.
\end{remark}

\begin{figure}
 \caption{Coefficients $a_{I,J}$ in $\I2Gr(2,6)$}
  \label{Figure3}
\begin{center}
 \begin{tikzpicture}
    \tikzstyle{every node}=[draw,circle,fill=white,minimum size=4pt,
                            inner sep=0pt]

    \draw (0,0) node (32) [label=left:$(3\mbox{,}2)$] {}
        -- ++(270:1.5cm) node (31) [label=left:$(3\mbox{,}1)$] {}
        -- ++(225:2cm) node (3-1) [label=left:$(3\mbox{,}-1)$] {}
        -- ++(210:3cm) node (3-2) [label=left:$(3\mbox{,}-2)$] {}
         ++(330:3cm) node (-1-2) [label=left:$(-1\mbox{,}-2)$] {}
        -- ++(315:2cm) node (-1-3) [label=left:$(-1\mbox{,}-3)$] {}
        -- ++(270:1.5cm) node (-2-3) [label=left:$(-2\mbox{,}-3)$] {};

    \draw (31) -- ++(315:2cm) node (21) [label=left:$(2\mbox{,}1)$] {};
    \draw (21) -- ++(330:3cm) node (2-3) [label=right:$(2\mbox{,}-3)$] {};
    \draw (21) -- ++(270:1.5cm) node (2-1) [label=left:$(2\mbox{,}-1)$] {};
    \draw (2-3) -- ++(210:3cm) node (1-3) [label=right:$(1\mbox{,}-3)$] {};
    \draw (-1-2) -- ++(90:1.5cm) node (1-2) [label=right:$(1\mbox{,}-2)$] {};




    \draw (3-2) -- (1-2);
    \draw (3-1) -- (2-1);
    \draw (21) -- (1-2);
    \draw (2-1) -- (-1-2);
    \draw (1-3) -- (-1-3);
    \draw (2-1) -- (2-3);
    \draw (3-1) -- (1-2);
    \node[draw=none] at (2.4,-4.4) {$(\epsilon_1-\epsilon_2)$};
     \draw (3-1) -- (2-3);
     \draw (2-1) -- (1-3);
    \node[draw=none] at (-1.4,-3.7) {2};
    \node[draw=none] at (1.7,-3.7) {2};
    \node[draw=none] at (0,-0.75) {1};
    \node[draw=none] at (-0.7,-2.2) {1};
    \node[draw=none] at (0.7,-2.2) {1};
    \node[draw=none] at (-0.7,-4) {1};
    \node[draw=none] at (1.4,-4) {1};
    \node[draw=none] at (0.7,-4) {1};
    \node[draw=none] at (-2.8,-3.7) {1};
    \node[draw=none] at (2.8,-3.7) {1};   
    \node[draw=none] at (-0.7,-5.5) {1};    
    \node[draw=none] at (0.7,-5.5) {1};
    \node[draw=none] at (2.8,-5.1) {1};
    \node[draw=none] at (-0.7,-6.6) {1};    
    \node[draw=none] at (0.7,-6.6) {1};
    \node[draw=none] at (0,-8) {1};
    \node[draw=none] at (-2.6,-4.4) {$(\epsilon_2-\epsilon_3)$};
    \draw (1-2) -- (1-3);
    \draw (3-2) -- (1-3);
    \node[draw=none] at (-2.8,-3.7) {1}; 
    \node[draw=none] at (-1.4,-5.5) {1};
    \node[draw=none] at (1.4,-5.1) {1};
    \node[draw=none] at (-2,-4.95) {1};


\end{tikzpicture}
\end{center}
\end{figure}
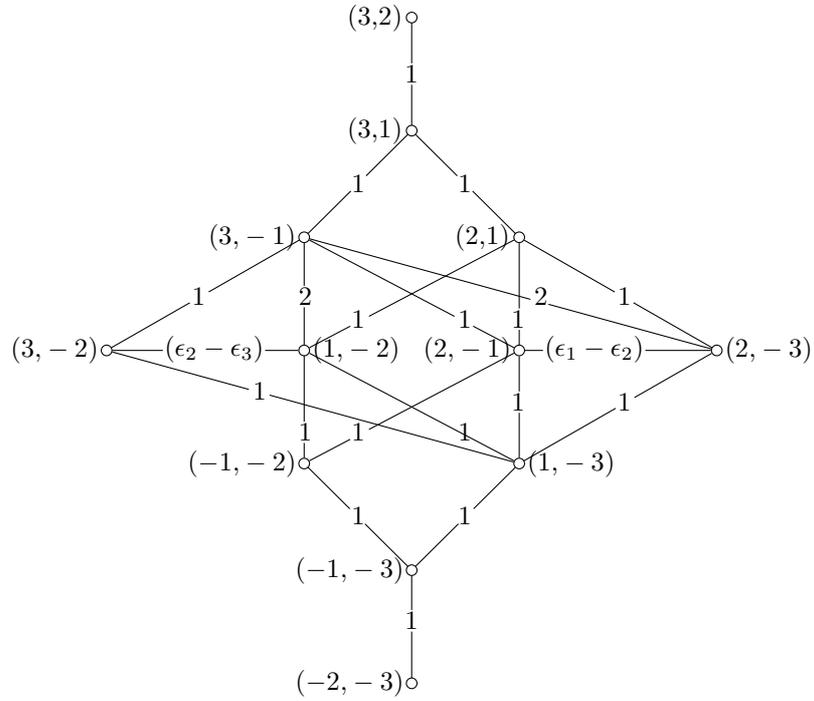

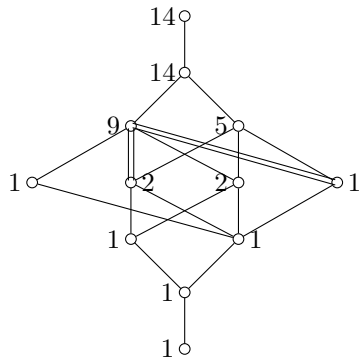
\begin{figure}
 \caption{Degrees of Schubert varieties inside $\I2Gr(2,6)$}
  \label{Figure4}
\begin{center}
 \begin{tikzpicture}
    \tikzstyle{every node}=[draw,circle,fill=white,minimum size=4pt,
                            inner sep=0pt]

    \draw (0,0) node (32) [label=left:$14$] {}
        -- ++(270:0.75cm) node (31) [label=left:$14$] {}
        -- ++(225:1cm) node (3-1) [label=left:$9$] {}
        -- ++(210:1.5cm) node (3-2) [label=left:$1$] {}
         ++(330:1.5cm) node (-1-2) [label=left:$1$] {}
        -- ++(315:1cm) node (-1-3) [label=left:$1$] {}
        -- ++(270:0.75cm) node (-2-3) [label=left:$1$] {};

    \draw (31) -- ++(315:1cm) node (21) [label=left:$5$] {};
    \draw (21) -- ++(330:1.5cm) node (2-3) [label=right:$1$] {};
    \draw (21) -- ++(270:0.75cm) node (2-1) [label=left:$2$] {};
    \draw (2-3) -- ++(210:1.5cm) node (1-3) [label=right:$1$] {};
    \draw (-1-2) -- ++(90:0.75cm) node (1-2) [label=right:$2$] {};




    \draw (3-1) -- (2-1);
    \draw (21) -- (1-2);
    \draw (2-1) -- (-1-2);
    \draw (1-3) -- (-1-3);
    \draw (2-1) -- (1-3);
    \draw (-0.74,-1.49) -- (-0.74,-2.18);
    \draw (-0.67,-1.49) -- (-0.67,-2.18);
    \draw (-0.685,-1.48) -- (1.98,-2.24);
    \draw (-0.685,-1.42) -- (1.98,-2.17);
    \draw (1-2) -- (1-3);
    \draw (3-2) -- (1-3);


\end{tikzpicture}
\end{center}
\end{figure}

From the equivariant cohomology, one can recover the classical cohomology of $\I2Gr(2,6)$ (Theorem \ref{thmeqclasscohomgen}). We will use the following notations: 
\[
\sigma_1:=\sigma_{3,1} \mbox{ , }\sigma_2:=\sigma_{2,1} \mbox{ , }\sigma_3:=\sigma_{3,-2} \mbox{ , }\sigma_3':=\sigma_{2,-3},
\]
with 
\[
\deg(\sigma_1)=14 \mbox{ , }\deg(\sigma_2)= 5\mbox{ , }\deg(\sigma_3)= 1\mbox{ , }\deg(\sigma_3')= 1.
\]

\begin{theorem}
\label{thm2cohok2n6}
A presentation of the cohomology of the bisymplectic Grassmannian $\I2Gr(2,6)$ is given by: 
\[
\HHH^*(\I2Gr(2,6),\ZZ)\cong \ZZ[\sigma_1,\sigma_2,\sigma_3,\sigma_3']/I,
\]
where $I$ is the ideal generated by the following elements:
\[
\begin{array}{cccccc}
2\sigma_1^4-2\sigma_1^2\sigma_2-3\sigma_1\sigma_3'  &  &  ,&  &  \sigma_2\sigma_3'&  ,\\
 \sigma_1\sigma_3-\sigma_1\sigma_3' &  &  ,&  &  \sigma_3\sigma_3'-\sigma_1^3\sigma_3'&  ,\\
 \sigma_2^2-\sigma_1^4+2\sigma_1^2\sigma_2+2\sigma_1\sigma_3' &  &  ,&  &  \sigma_3^2&  ,\\
 \sigma_1^5-14\sigma_1^2\sigma_3' &  &  ,&  &  \sigma_3'^2&  ,\\
 \sigma_2\sigma_3 &  &  ,&  & \sigma_1^4\sigma_3'&  .\\
 \end{array}
\]
\end{theorem}

\begin{proof}
\enlargethispage{\baselineskip}
First, we prove that $\sigma_1,\sigma_2,\sigma_3,\sigma_3'$ generate the cohomology by showing that they generate all the Schubert classes $\sigma_I$. This is a consequence of the following formulas, which can be derived directly from Figure \ref{Figure4}:
\[
\begin{array}{c}
\sigma_{(3,-1)}=\sigma_1^2-\sigma_2,\\
\sigma_{(2,-1)}= 3\sigma_1\sigma_2-\sigma_1^3+\sigma_3,\\
\sigma_{(1,-2)}= \sigma_1^3-2\sigma_1\sigma_2-\sigma_3-\sigma_3',\\
\sigma_{(-1,-2)}= \sigma_1^4-2\sigma_1^2\sigma_2-3\sigma_1\sigma_3',\\
\sigma_{(1,-3)}=\sigma_1\sigma_3' ,\\
\sigma_{(-1,-3)}=\sigma_1^2\sigma_3' ,\\
\sigma_{(-2,-3)}=\sigma_1^3\sigma_3' .\\
\end{array}
\]
The relations generating $I$ involving the product of $\sigma_1$ with other classes can be derived from Figure \ref{Figure4} too. For the remaining relations, they can be derived from the following identities, which hold in the equivariant cohomology, and can be verified by computing explicitly the classes $\sigma_I$:
\[
\begin{array}{c}
\sigma_2^2=\sigma_2(\epsilon_3-\epsilon_1)(\epsilon_3-\epsilon_2)+\sigma_{(1,-2)}(\epsilon_3-\epsilon_1)+\sigma_{(2,-1)}(\epsilon_3-\epsilon_2)+\\
+\sigma_3'(\epsilon_3-\epsilon_2)+ \sigma_1\sigma_{(1,-2)},\\
\sigma_2\sigma_3=(\epsilon_2+\epsilon_3)(\sigma_{(1,-2)}(\epsilon_2-\epsilon_3)+\sigma_{(1,-3)}) ,\\
\sigma_2\sigma_3'=2\epsilon_3(\sigma_3'(\epsilon_3-\epsilon_2)+\sigma_{(1,-3)}) ,\\
\sigma_3\sigma_3'=\sigma_{(-2,-3)} ,\\
\sigma_3^2=2\epsilon_2(\sigma_3(\epsilon_1+\epsilon_2)(\epsilon_2-\epsilon_1)+\sigma_{(1,-2)}(\epsilon_1+\epsilon_3)(\epsilon_3-\epsilon_2)+\\
-\sigma_{(-1,-2)}(\epsilon_3-\epsilon_2)-\sigma_{(1,-3)}(\epsilon_1+\epsilon_3)+\sigma_{(-1,-3)}) ,\\
 \sigma_3'^2=2\epsilon_3(\sigma_3'(\epsilon_3-\epsilon_1)(\epsilon_3+\epsilon_1)+\sigma_{(1,-3)}(\epsilon_1+\epsilon_2)-\sigma_{(-1,-3)}) .
\end{array}
\]
We have verified that these are all the relations inside $I$ by showing that they generate all products involving $\sigma_1,\sigma_2,\sigma_3,\sigma_3'$.
\end{proof}

\begin{remark}
\label{rempoincarebisym}
The basis given by the Schubert classes inside $\I2Gr(2,6)$ is not self-dual with respect to the intersection product. For instance, the non zero products of codimension $3$ Schubert classes are as follows:
\[
\begin{array}{c}
\sigma_{(3,-2)}\sigma_{(2,-3)}=1 ,\\
\sigma_{(1,-2)}\sigma_{(2,-1)}=1 ,\\
\sigma_{(3,-2)}\sigma_{(2,-1)}=-1 .
\end{array}
\]
A self-dual basis in codimension $3$ would be given by $\sigma_{(3,-2)}, \sigma_{(2,-3)}, \sigma_{(1,-2)}, \sigma_{x}=\sigma_{(2,-1)}+\sigma_{(2,-3)}$. In this basis, the degree diagram is the one shown in Figure \ref{Figure5}. Notice that the diagram is symmetric with respect to a central reflection; this is a consequence of the fact that the additive basis chosen is self-dual.
\end{remark}

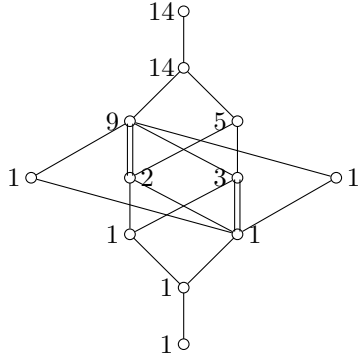
\begin{figure}
 \caption{Degree of classes in a self-dual basis; the codimension $3$ classes are, from left to right: $\sigma_{(3,-2)},  \sigma_{(1,-2)}, \sigma_{x},\sigma_{(2,-3)}$}
  \label{Figure5}
\begin{center}
 \begin{tikzpicture}
    \tikzstyle{every node}=[draw,circle,fill=white,minimum size=4pt,
                            inner sep=0pt]

    \draw (0,0) node (32) [label=left:$14$] {}
        -- ++(270:0.75cm) node (31) [label=left:$14$] {}
        -- ++(225:1cm) node (3-1) [label=left:$9$] {}
        -- ++(210:1.5cm) node (3-2) [label=left:$1$] {}
         ++(330:1.5cm) node (-1-2) [label=left:$1$] {}
        -- ++(315:1cm) node (-1-3) [label=left:$1$] {}
        -- ++(270:0.75cm) node (-2-3) [label=left:$1$] {};

    \draw (31) -- ++(315:1cm) node (21) [label=left:$5$] {};
    \draw (21) -- ++(270:0.75cm) node (2-1) [label=left:$3$] {};
    \draw (-1-3) -- ++(45:1cm) node (1-3) [label=right:$1$] {};
    \draw (1-3) -- ++(30:1.5cm) node (2-3) [label=right:$1$] {};
    \draw (-1-2) -- ++(90:0.75cm) node (1-2) [label=right:$2$] {};




    \draw (3-1) -- (2-1);
    \draw (21) -- (1-2);
    \draw (2-1) -- (-1-2);
    \draw (3-1) -- (2-3);
    \draw (-0.74,-1.49) -- (-0.74,-2.18);
    \draw (-0.67,-1.49) -- (-0.67,-2.18);
    \draw (0.74,-2.23) -- (0.74,-2.92);
    \draw (0.67,-2.23) -- (0.67,-2.92);
    \draw (1-2) -- (1-3);
    \draw (3-2) -- (1-3);


\end{tikzpicture}
\end{center}
\end{figure}

\begin{remark}
\label{rmsymmetriesbisym}
The group of permutations $\mathfrak{S}_n$ acts on the cohomology of the bisymplectic Grassmannians, even though it does not act on the varieties themselves; the action is a consequence of a monodromy phenomenon.

Let $X$ be a bisymplectic Grassmannian $\I2Gr(k,2n)$ defined by the forms 
\[
\omega_1=\sum_{i=1}^{n} x_i\wedge x_{-i} \mbox{ and }\omega_2=\sum_{i=1}^{n} \lambda_i x_i\wedge x_{-i}.
\]
Let $\eta$ be an element of the group of permutations $\mathfrak{S}_n$. There exists a curve $\gamma$ inside the space of pencils of bisymplectic forms that goes from $\Omega=\langle \omega_1,\omega_2\rangle$ to $\eta.\Omega=\langle\omega_1,\eta.\omega_2\rangle$, where
\[
\eta.\omega_2=\sum_{i=1}^{n} \lambda_{\eta(i)} x_i\wedge x_{-i}.
\]
Following the curve, one obtains a continuous deformation $\gamma$ such that $\gamma(0)=X=\gamma(1)$, and which sends a Schubert variety $\sigma_I$ to $\eta.\sigma_I$, where the action on $\sigma_I$ is induced by the one of $\mathfrak{S}_n$ on the pencils. As the cohomology is locally constant, the action on Schubert varieties induces an action in cohomology. In the following we show concretely what it means in the case when $k=2,n=3$.

As the irreducible representations of $\mathfrak{S}_3$ given by Schubert classes with codimension different from $3$ are only $1$-dimensional, we will focus on codimension $3$ Schubert varieties. They admit the following explicit description:
\[
\alpha_2:=\sigma_{(3,-2)}= v_{-2} \wedge \PP(\langle v_{\pm 3},v_{\pm 1} \rangle),
\]
\[
\beta_1:=\sigma_{(1,-2)}=\{x\in \PP(\langle v_{-2},v_{-3}\rangle)\wedge \PP(\langle v_{\pm 1}, v_{-2}, v_{-3}\rangle) \mbox{ s.t. }x\neq 0\},
\]
\[
\beta_2:=\sigma_{(2,-1)}=\{x\in \PP(\langle v_{-1},v_{-3}\rangle)\wedge \PP(\langle v_{\pm 2}, v_{-1}, v_{-3}\rangle) \mbox{ s.t. }x\neq 0\},
\]
\[
\alpha_3:=\sigma_{(2,-3)}=v_{-3} \wedge \PP(\langle v_{\pm 2},v_{\pm 1} \rangle).
\]
Moreover, inside the cohomology of $\I2Gr(2,6)$ there are two more remarkable varieties:
\[
\alpha_1:= v_{-1} \wedge \PP(\langle v_{\pm 3},v_{\pm 2} \rangle),
\]
\[
\beta_3:=\{x\in \PP(\langle v_{-1},v_{-2}\rangle)\wedge \PP(\langle v_{\pm 3}, v_{-1}, v_{-2}\rangle) \mbox{ s.t. }x\neq 0\}.
\]
Actually, there are also varieties $\alpha_{-1}, \alpha_{-2}, \alpha_{-3}, \beta_{-1}, \beta_{-2}, \beta_{-3}$, but one can prove that in cohomology $\alpha_{i}=\alpha_{-i}$ and $\beta_{i}=\beta_{-i}$ for $i=1,2,3$. The action of $\mathfrak{S}_3$ on the $\alpha_i$'s and the $\beta_i$'s is the expected one. By using the products of the codimension $3$ Schubert varieties and the symmetries given by $\mathfrak{S}_3$, one can prove that
\[
\alpha_1-\alpha_2=\beta_1 - \beta_2,
\]
\[
\alpha_2-\alpha_3=\beta_2 - \beta_3.
\]
To summarize, the action of $\mathfrak{S}_3$ on $\HHH^i(\I2Gr(2,6),\ZZ)$ is trivial if $i\neq 6$, and $\HHH^6(\I2Gr(2,6),\ZZ)$ decomposes in the sum of two trivial representations generated by the classes of $\sigma_H^3=\alpha_2+3\beta_1+2\beta_2+3\alpha_3$ and $\sigma_{(2,1)}\sigma_H=\beta_1+\beta_2+\alpha_3$, and one natural $2$-dimensional representation given by the action on $\langle \alpha_1,\alpha_2,\alpha_3 \rangle$, with $\alpha_1+\alpha_2+\alpha_3=0$.
\end{remark}

\bibliographystyle{alpha}
\bibliography{bibliiotbisymGrass}

\makeatletter
\providecommand\@dotsep{5}
\makeatother

\listoftodos 

\end{document}